\documentclass[reqno]{amsart}

\usepackage{amsfonts,amsmath,amssymb,color,enumerate,amsthm,hyperref,bigints,dsfont,graphicx,tikz}
\usepackage[kerning=true]{microtype}

\usetikzlibrary{decorations.markings}
\usetikzlibrary{plotmarks}
\usetikzlibrary{patterns}

\numberwithin{equation}{section}

\newtheorem{theorem}{Theorem}[section]

\newtheorem{remark}[theorem]{Remark}
\newtheorem{definition}[theorem]{Definition}
\newtheorem{proposition}[theorem]{Proposition}
\newtheorem{lemma}[theorem]{Lemma}

\newcommand{\be}{\begin{equation*}}
\newcommand{\ee}{\end{equation*}}
\newcommand{\ben}{\begin{equation}}
\newcommand{\een}{\end{equation}}
\newcommand{\begincal}{\begin{eqnarray*}}
\newcommand{\fincal}{\end{eqnarray*}}
\newcommand{\bal}{\begin{aligned}}
\newcommand{\eal}{\end{aligned}}
\newcommand{\ds}{\displaystyle}

\DeclareMathOperator{\bigO}{O}
\DeclareMathOperator{\smallo}{o}
\DeclareMathOperator{\Scal}{Scal}
\DeclareMathOperator{\Weyl}{Weyl}
\DeclareMathOperator{\Vol}{Vol}

\newcommand{\loc}{loc}

\newcommand{\eps}{\varepsilon}
\newcommand{\R}{\mathbb{R}}
\renewcommand{\S}{\mathbb{S}}
\newcommand{\N}{\mathbb{N}}

\newcommand{\<}{\left<}
\renewcommand{\>}{\right>}
\renewcommand{\(}{\left(}
\renewcommand{\)}{\right)}

\newcommand{\pui}{\frac{n-2}{2}}

\newcommand{\vp}{\varphi}

\numberwithin{equation}{section}

\title[Sign-changing blow-up for the Yamabe equation]{Sign-changing blow-up for the Yamabe equation at the lowest energy level}

\author{Bruno Premoselli}

\address{Bruno Premoselli, Universit\'e Libre de Bruxelles, Service d'analyse, CP 218, Boulevard du Triomphe, B-1050 Bruxelles, Belgique.}
\email{bruno.premoselli@ulb.be}

\author{J\'er\^ome V\'etois}

\address{J\'er\^ome V\'etois, Department of Mathematics and Statistics, McGill University, 805 Sherbrooke Street West, Montreal, Quebec H3A 0B9, Canada}
\email{jerome.vetois@mcgill.ca}

\thanks{The first author was supported by the FNRS CdR grant J.0135.19, the Fonds Th\'elam and an ARC Avanc\'e 2020 grant. The second author was supported by the Discovery Grant RGPIN-2016-04195 from the Natural Sciences and Engineering Research Council of Canada.}

\date{June 16, 2022}

\begin{document}

\begin{abstract}
We investigate the blow-up behavior of sequences of sign-changing solutions for the Yamabe equation on a Riemannian manifold $(M,g)$ of positive Yamabe type. For each dimension $n\ge11$, we describe the value of the minimal energy threshold at which blow-up occurs. In dimensions $11 \le n \le 24$, where the set of \emph{positive} solutions is known to be compact, we show that the set of \emph{sign-changing} solutions is not compact and that blow-up already occurs at the lowest possible energy level. We prove this result by constructing a smooth, non-locally conformally flat metric on space forms $\S^n/\Gamma$, $\Gamma \neq \{1\}$, whose Yamabe equation admits a family of sign-changing blowing-up solutions. As a counterpart of this result, we also prove a sharp compactness result for sign-changing solutions at the lowest energy level, in small dimensions or under strong geometric assumptions. 
\end{abstract}

\maketitle

\section{Introduction}\label{Intro}

\subsection{Introduction and statements of the main results} 

Let $\(M,g\)$ be a smooth, closed (i.e. compact and without boundary) Riemannian manifold of dimension $n\ge3$. In this paper, we are interested in the existence of sequences of sign-changing blowing-up solutions $\(u_k\)_{k\in\N}$ in $C^2\(M\)$ to the nodal (or sign-changing) Yamabe equation
\begin{equation}\label{IntroEq1}
\Delta_gu_k+c_n\Scal_gu_k=\left|u_k\right|^{2^*-2}u_k\quad\text{in }M,
\end{equation}
where $\Delta_g:=-\text{div}_g\nabla$ is the Laplace--Beltrami operator, $c_n:=\frac{n-2}{4\(n-1\)}$, $\Scal_g$ is the scalar curvature of the manifold and $2^* =\frac{2n}{n-2}$ is the critical exponent for the embeddings of the Sobolev space $H^1\(M\)$ into Lebesgue's spaces. Solutions of \eqref{IntroEq1} are in $C^{3,\alpha}(M)$ for $0 < \alpha < \min(2^*-2,1)$ by Trudinger's result \cite{Trudinger} and standard elliptic theory. We recall that $\(u_k\)_{k\in\N}$ is said to \emph{blow up} as $k \to \infty$ if $\left\| u_k \right\|_{L^\infty(M)}\to\infty$ as $k\to\infty$. 

\smallskip
The \emph{Yamabe invariant} of the conformal class $[g]$ is defined as 
$$\bal 
Y(M,[g])&:=\inf_{\hat{g}\in [g]}\(\Vol_{\hat{g}}\(M\)^{\frac{2-n}{n}}\int_M\Scal_{\hat{g}}dv_{\hat{g}}\) \\
& =\frac{4(n-1)}{n-2}\cdot\inf_{u \in C^\infty(M)} \frac{\int_M\(|\nabla u|_g^2 + c_n \Scal_g u^2\)dv_g}{\( \int_M |u|^{2^*} dv_g \)^{\frac{2}{2^*}}},
\eal $$
where $\Vol_{\hat{g}}\(M\)$ is the volume of $\(M,\hat{g}\)$. Letting $\(\S^n,g_{std}\)$ be the standard unit sphere of dimension $n\ge3$, by conformal invariance, we also have
 $$ Y(\mathbb{S}^n,[g_{std}]) =  \inf_{u \in C^\infty_c(\R^n)} \frac{\int_{\R^n} |\nabla u|^2 dx }{\( \int_{\R^n} |u|^{2^*} dx \)^{\frac{2}{2^*}}}, $$
so that $Y(\mathbb{S}^n,[g_{std}])^{-\frac{1}{2}}$ is the optimal constant for the Sobolev inequality in $\R^n$. We say that $(M,g)$ is of \emph{positive Yamabe type} if $Y(M,[g])>0$, i.e. $ \triangle_{g} + c_n \Scal_{g} $ is coercive. In this case, when $(M,g)$ is not conformally diffeomorphic to the standard sphere $(\mathbb{S}^n, g_{std})$ (which we denote in what follows by $ (M,g) \not \approx (\mathbb{S}^n, g_{std})$), we have $Y(M,[g]) < Y(\mathbb{S}^n,[g_{std}])$, and the existence of a positive solution $u_0$ to the Yamabe equation
\begin{equation}\label{IntroEq4}
\Delta_gu_0+c_n\Scal_g u_0=u_0^{2^*-1}\quad\text{in }M
\end{equation}
attaining the Yamabe invariant $Y(M,[g])$ is known since the work of Trudinger \cite{Trudinger}, Aubin \cite{Aubin2} and Schoen \cite{SchoenYamabe}. 

\smallskip
In this paper, we are interested in determining the value of the minimal energy level at which \emph{sign-changing} blow-up occurs for \eqref{IntroEq1}. For every $u \in H^1(M)$, we define the energy of $u$ as 
$$E\(u\):=\int_M\left|u\right|^{2^*}dv_g.$$ 
We define $I(M, [g]) \subset (0, \infty]$ as the set of numbers $E\in (0, + \infty]$ such that \eqref{IntroEq1} admits a \emph{blowing-up} sequence of solutions $(u_k)_{k\in\N}$  with $\limsup_{k \to\infty} E(u_k) = E$. We then define
\[ E(M, [g]):= \inf I(M, [g]). \]
Similarly, we let $I_+(M, [g]) \subset (0,+ \infty]$ be the set of numbers $E\in(0, + \infty]$ such that \eqref{IntroEq1} admits a \emph{blowing-up} sequence of \emph{positive} solutions $(u_k)_{k\in\N}$  satisfying $\limsup_{k \to \infty} E(u_k) = E$, and we define
\[ E_+(M, [g]):= \inf I_+(M, [g]). \]
The value $+ \infty$ is allowed in the definitions of $E(M,[g])$ and $E_+(M, [g])$ and corresponds to sequences of solutions of \eqref{IntroEq1} with diverging energies. When $I(M,[g])$ (resp. $I_+(M,[g])$) is empty, it means that \eqref{IntroEq1} does not admit any blowing-up solutions (resp. positive blowing-up solutions), and thus the set of solutions (resp. positive solutions) of \eqref{IntroEq1} is compact in $C^2(M)$ by standard elliptic theory. In this case, we let $E(M,[g]) := -\infty$ (resp  $E_+(M,[g]): = -\infty$). If $I(M,[g]) \neq \emptyset$ and $(u_k)_{k\in\N}$ is a sequence of solutions of \eqref{IntroEq1} satisfying $\limsup_{k \to \infty} E(u_k) < E(M, [g])$, then by definition $(u_k)_{k \in \mathbb{N}}$ does not blow-up, and is thus precompact in $C^2(M)$. 

\smallskip
As long as \emph{positive} solutions are considered, $E_+(M, [g])$ is well understood: in the proof of the compactness of positive solutions of the Yamabe equation, assuming the validity of the Positive Mass Theorem when $n \ge 8$,  Khuri--Marques--Schoen \cite{KhuMaSc} (with previous contributions by Schoen \cite{SchoenPreprint,Schoenlcf}, Li--Zhu \cite{LiZhu}, Druet \cite{DruetYlowdim}, Marques \cite{Marques} and Li--Zhang \cite{LiZhang1,LiZhang2}) proved that
\begin{equation} \label{poscomp}
  E_+(M, [g]) =  -\infty \text{ when } \left\{ \begin{aligned} & 3 \le n \le 24 \text{ and } (M,g) \not \approx (\mathbb{S}^n, g_{std}), \\ & (M,g) \text{ is locally conformally flat or } \\  &n \ge 6 \text{ and }  \sum_{k=0}^{[\frac{n-6}{2}]}|\nabla^k\Weyl_g(x)|^2 > 0 \; \;\forall x \in M, \end{aligned} \right. 
  \end{equation}
where $\Weyl_g$ is the Weyl curvature tensor of the manifold. On the other side, Brendle \cite{Brendle} and Brendle--Marques \cite{BrendleMarques} proved that there exists a non-locally conformally flat metric $g$ on $\S^n$ such that 
\begin{equation} \label{poscomp2}
 E_+(\S^n, [g]) = Y(\mathbb{S}^n,[g_{std}])^{\frac{n}{2}}  \quad \text{ when } n \ge 25 .
\end{equation}

\smallskip
The blow-up behavior of \emph{sign-changing} solutions of \eqref{IntroEq1} is far less understood. A simple application of Struwe's celebrated $H^1$-compactness result \cite{Struwe} (see Proposition \ref{Pr1} below for a proof) shows that if $(M, g)$ is of positive Yamabe type, then any \emph{sign-changing} blowing-up sequence $(u_k)_{k\in\N}$ of solutions of \eqref{IntroEq1} satisfies
\begin{equation} \label{IntroStruwe}
\liminf_{k \to  \infty}E\(u_k\)\ge Y(\mathbb{S}^n,[g_{std}])^{\frac{n}{2}} + Y(M,[g])^{\frac{n}{2}}. 
\end{equation}
 As a consequence, if $g$ is the Brendle \cite{Brendle} and Brendle--Marques \cite{BrendleMarques} metric, then, by \eqref{poscomp2}, we obtain
$$ E(\S^n, [g]) = E_+(\S^n, [g]) = Y(\mathbb{S}^n,[g_{std}])^{\frac{n}{2}}  \quad \text{ when } n \ge 25. $$
In dimensions $n\ge 25$, blow-up for \eqref{IntroEq1} thus occurs at the lowest energy level for positive solutions. Our aim in this paper is to investigate the situation for \emph{sign-changing} solutions and in particular, the value of $E(M, [g])$ in dimensions $n\le24$, where the set of \emph{positive} solutions of \eqref{IntroEq1} is compact, i.e. $E_+(M, [g])=-\infty$. As follows from \eqref{IntroStruwe}, a lower bound on $E(M, [g])$ is given by $Y(\mathbb{S}^n,[g_{std}])^{\frac{n}{2}} + Y(M,[g])^{\frac{n}{2}}$. Our main result shows that this lower bound is attained in dimensions $11 \le n \le 24$:

\begin{theorem}\label{Th0}
Assume that $11 \le n\le24$ and let $\Gamma$ be a finite subgroup of isometries of $(\mathbb{S}^n, g_0)$, $\Gamma \neq \{Id\}$, acting freely and smoothly on $\mathbb{S}^n$. There exists a smooth, non-locally conformally flat Riemannian metric $g$ on $\mathbb{S}^n /\Gamma$ of positive Yamabe type and such that
$$ E(M,[g]) = Y(\mathbb{S}^n,[g_{std}])^{\frac{n}{2}}+Y(\mathbb{S}^n /\Gamma,[g])^{\frac{n}{2}}. $$
\end{theorem}

The metric $g$ in Theorem \ref{Th0} is not the quotient metric $g_{\Gamma}$ but can be chosen arbitrarily close to it. Theorem \ref{Th0}  is a special case of a more general result, Theorem \ref{Th1} below, which holds true in any dimension $n \ge 11$ and for a larger class of manifolds than the spherical space forms $\mathbb{S}^n/\Gamma$. We refer to Section \ref{secmainth} for more details on this regard. 

\smallskip 
By \eqref{poscomp}, the set of \emph{positive} solutions of \eqref{IntroEq1} on $(\S^n/\Gamma, g)$, where $g$ is given by Theorem \ref{Th0}, is compact in $C^2(M)$. As Theorem \ref{Th0} shows, however, in this case, the set of \emph{sign-changing} solutions is not compact and blow-up already occurs at the minimal energy level. This phenomenon of loss of compactness for sign-changing solutions in situations where the set of positive solutions is compact was recently highlighted in Premoselli--V\'etois \cite{PremoselliVetois2} for critical Schr\"odinger-type equations in $M$. We prove Theorem \ref{Th0} by constructing a sign-changing blowing-up sequence $(u_k)_{k\in\N}$ of solutions of \eqref{IntroEq1} such that
\begin{equation}\label{IntroEq5}
\lim_{k \to \infty}E\(u_k\)=Y(\mathbb{S}^n,[g_{std}])^{\frac{n}{2}} + Y(M,[g])^{\frac{n}{2}}.
\end{equation}

\smallskip 
As a counterpart of Theorem \ref{Th0}, we also prove the following result, which provides a lower bound for $E(M,[g])$ in smaller dimensions or under strong geometric assumptions:

\begin{theorem}\label{Th2}
Let $\(M,g\)$ be a smooth, closed Riemannian manifold of dimension $n\ge3$ and positive Yamabe type which is not conformally diffeomorphic to the standard sphere $(\mathbb{S}^n,[g_{std}])$. Assume that one of the following conditions is satisfied: 
\begin{itemize}
\item $(M,g)$ is locally conformally flat, 
\item $ n\le 9$, 
\item $n=10$ and $u_0\ne\frac{5}{567}|\Weyl_g|^2_g$ for all points in $M$ and all solutions $u_0$ of \eqref{IntroEq4} attaining $Y(M,[g])$, or 
\item $n\ge11$ and $\Weyl_g\ne0$ for all points in $M$.
\end{itemize}
Then
\[ E(M, [g]) > Y(\mathbb{S}^n,[g_{std}])^{\frac{n}{2}}+Y(M,[g])^{\frac{n}{2}}. \]
\end{theorem}

Under the assumptions of Theorem \ref{Th2} and by \eqref{poscomp}, the set of \emph{positive} solutions of \eqref{IntroEq1} is compact in $C^2(M)$, and thus blow-up for \eqref{IntroEq1} can only occur for \emph{sign-changing} solutions, but at an energy level that is strictly higher than the minimal one given by \eqref{IntroStruwe}. In particular, if the assumptions of Theorem \ref{Th2} are satisfied, then there exists a constant $\eps_0>0$ such that the set of all solutions $u\in C^2\(M\)$ of \eqref{IntroEq1} satisfying
$$E\(u\)\le Y(M,[g])^{\frac{n}{2}}+Y(\mathbb{S}^n,[g_{std}])^{\frac{n}{2}}+\eps_0$$
is compact in $C^2\(M\)$. The contrapositive of Theorem \ref{Th2} provides necessary conditions for sign-changing blowing-up solutions to exist at the minimal energy level. In dimensions $n\ge11$, this condition is that $\Weyl_g$ vanishes at some (but not all) points in $M$. This is consistent with Theorem \ref{Th0} since the sequence $(u_k)_{k\in\N}$ that we construct to prove Theorem \ref{Th0} blows up at a point where $\Weyl_g$ vanishes. Hence, Theorems \ref{Th0} and \ref{Th2} are sharp in every dimension $11 \le n \le 24$. 
The condition arising when $n=10$ is purely analytical (see \eqref{Th3Eq5} below). 
The exact value of $E(M, [g])$ when $3 \le n \le 10$ is not yet known: as Theorem \ref{Th2} shows, at least when $n \le 9$, it is strictly larger than $Y(M,[g])^{\frac{n}{2}}+Y(\mathbb{S}^n,[g_{std}])^{\frac{n}{2}} $. Computing $E(M,[g])$ when $3 \le n \le 10$ will be the focus of forthcoming work. 

\smallskip
We conclude this subsection by mentioning an important additional motivation for investigating the value of $E(M,[g])$ and, more generally, (non-)compactness issues for sign-changing solutions of \eqref{IntroEq1}. 
Let $g$ be a Riemannian metric in $M$ of positive Yamabe type. For $k \in\N$ and $\tilde g \in [g]$, we denote by $\lambda_k(\tilde g)$ the $k$-th eigenvalue (counted with multiplicity) of the conformal Laplacian $\triangle_{\tilde g} + c_n \Scal_{\tilde g}$ in $M$. Ammann--Humbert introduced in \cite{AmmannHumbert}  the so-called $k$-th Yamabe invariant:
\[\mu_k(M, [g]) = \inf_{\tilde g \in [g]} \lambda_k(\tilde g) \Vol_{\hat{g}}\(M\)^{\frac{2}{n}}. \]
In the case where $k=2$, test functions computations show that
\begin{equation} \label{propmu2}
\mu_2(M,[g])^{\frac{n}{2}} \le Y(M,[g])^{\frac{n}{2}}+Y(\mathbb{S}^n,[g_{std}])^{\frac{n}{2}}
\end{equation}
holds. For $k \ge 2$, extremal metrics attaining $\mu_k(M,[g])$, when they exist, are not smooth in general. When $(M,g)$ is of positive Yamabe type,  Amman--Humbert \cite{AmmannHumbert} established the existence of extremal metrics attaining $\mu_2(M,[g])$ provided $(M,g)$ is not locally conformally flat and $n\ge11$. Moreover, Amman--Humbert \cite{AmmannHumbert} obtained that if $\mu_2(M,[g])$ is attained by a generalized metric $\tilde g = u^{\frac{4}{n-2}} g$ with $u \ge 0$ and $\Vert u \Vert_{L^{2^*}(M)} = 1$, then there exists a generalized eigenvector $\tilde w$ associated to $\mu_2(M,[g])$ such that $u = |\tilde w|$ and $w:= \mu_2(M,[g])^{\frac{n-2}{4}} \tilde w$ is a sign-changing solution of \eqref{IntroEq1} satisfying $E(w) = \mu_2(M,[g])^{\frac{n}{2}}$. With \eqref{propmu2}, this shows that extremal metrics for $\mu_2(M,[g])$ give rise to sign-changing solutions of \eqref{IntroEq1} whose energies lie below the minimal blow-up level given by \eqref{IntroStruwe}. Theorem \ref{Th2} thus provides a compactness result for a range of energy levels including $\mu_2(M,[g])$. Such a result is generally perceived as a strong indication that $\mu_2(M,[g])$ is attained (at least for analogous problems in the two-dimensional case, see for example Matthiesen--Siffert \cite{MatthiesenSiffert} or P\'etrides \cite{Petrides3}): Theorem \ref{Th2} can therefore also be seen as a first step in a more systematic investigation of $\mu_k(M,[g])$. 
 
\subsection{Review of the literature and outline of the paper}

Existence results for the nodal Yamabe equation \eqref{IntroEq1} have been the subject of several work in the last decades. On the standard sphere $(\mathbb{S}^n, g_{std})$, existence results of large-energy sign-changing solutions of \eqref{IntroEq1} are in Ding \cite{Ding}, del Pino--Musso--Pacard--Pistoia \cite{DelPinoMussoPacardPistoia1,DelPinoMussoPacardPistoia2}, Musso--Wei \cite{MussoWei}, Medina--Musso--Wei \cite{MedinaMussoWei} and Medina--Musso \cite{MedinaMusso}; other existence results at lower energy levels are in  Clapp \cite{Clapp} and Fernandez--Petean \cite{FernandezPetean}. For more general manifolds, existence and multiplicity results of sign-changing solutions of \eqref{IntroEq1} have been obtained by Ammann--Humbert \cite{AmmannHumbert}, V\'etois \cite{Vetois}, Clapp--Fern\'andez \cite{ClappFernandez}, Clapp--Pistoia--Tavares \cite{ClappPistoiaTavares} and Gursky--P\'erez-Ayala \cite{GurskyPerez}. Multiplicity results of sign-changing solutions of Yamabe--Schr\"odinger-type equations with more general potential functions can also be found in V\'etois \cite{Vetois} and Clapp--Fern\'andez \cite{ClappFernandez}. 

\smallskip
Theorem \ref{Th0} is both a non-compactness and an existence result: it shows in particular the existence of infinitely many solutions of \eqref{IntroEq1} on $(\S^n/\Gamma, g)$. Theorem \ref{Th2}, on the contrary, is a compactness result for \eqref{IntroEq1} below the energy level $Y(\mathbb{S}^n,[g_{std}])^{\frac{n}{2}}+Y(M,[g])^{\frac{n}{2}}$. Compactness and non-compactness results for sign-changing solutions of Yamabe--Schr\"odinger-type equations have been obtained by V\'etois \cite{Vetois} and recently by Premoselli--V\'etois \cite{PremoselliVetois1,PremoselliVetois2} (see also Robert--V\'etois \cite{RobertVetois3,RobertVetois4}, Pistoia--V\'etois \cite{PistoiaVetois} and Deng--Musso--Wei \cite{DengMussoWei} for existence results of sign-changing blowing-up solutions to equations of type \eqref{IntroEq1} with asymptotically critical nonlinearities). A general pointwise description of finite-energy blowing-up sequences of solutions of such equations, including the geometric case of \eqref{IntroEq1}, has recently been obtained by Premoselli \cite{Premoselli13}. To the best of the authors' knowledge, Theorem \ref{Th0} is the first constructive result of sign-changing \emph{blowing-up} solutions for the geometric equation \eqref{IntroEq1} on a different manifold than the standard sphere. 

\smallskip
The paper is organised as follows. In Section \ref{secmainth}, we state Theorem \ref{Th1} which is a generalization of Theorem \ref{Th0} in dimensions $n \ge 11$. We then prove Theorem \ref{Th1} in Sections \ref{secLS} and \ref{secreduced}. The proof relies on a constructive Lyapunov--Schmidt reduction method. Our approach is inspired from the constructions on the  sphere by Brendle \cite{Brendle} and Brendle--Marques \cite{BrendleMarques} (see also Ambrosetti--Malchiodi \cite{AmbrosettiMalchiodi2} and Berti--Malchioldi \cite{BertiMalchiodi}). In Section \ref{secLS}, we perform the Lyapunov--Schmidt reduction and construct a blowing-up sequence of approximate solutions of \eqref{IntroEq1} of the form $u_k = u_0 - B_k +$  lower order terms, where $u_0 >0$ solves \eqref{IntroEq1} and $B_k$ is a bubbling profile modeled on the positive standard bubble (see \eqref{defW} below). In Section \ref{secreduced}, we reduce the proof of Theorem \ref{Th1} to finding a critical point of an energy function in $\R^{n+1}$ (see \eqref{defF} below). The main difference with the constructions of Brendle \cite{Brendle} and Brendle--Marques \cite{BrendleMarques} for positive solutions  is that the critical point of $F$ that we find is of saddle-type: this allows us to conclude up to dimension $11$ but in turn forces us to work with greater precision and expand the reduced energy to the fourth order (see \eqref{Forder4} below). Finally, we prove Theorem \ref{Th2} in Section \ref{Sec:linearpert} by using a Pohozaev-type identity together with the pointwise blow-up description for sign-changing solutions of \eqref{IntroEq1} recently obtained by Premoselli \cite{Premoselli13}, which we refine here by using an approach based on iterated estimates, in the spirit of the method introduced by Chen--Lin \cite{ChenLin} in the case of positive solutions (see also Marques \cite{Marques} for applications to the Yamabe equation). 

\section{$Y$-non-degenerate metrics and a refined version of Theorem \ref{Th0}} \label{secmainth}

Let $(M,g)$ be a smooth, closed Riemannian manifold of dimension $n \ge 3$ and positive Yamabe type. By the resolution of the Yamabe problem (see Trudinger \cite{Trudinger}, Aubin \cite{Aubin2} and Schoen \cite{SchoenYamabe}) there exists a smooth positive function $u_0$ in $M$ that attains $Y(M,[g])$ and solves \eqref{IntroEq4}.
In particular $\int_M u_0^{2^*} dv_g = Y(M,[g])^{\frac{n}{2}}$. Let $\vp \in C^{\infty}(M), \vp >0$. By the conformal invariance property of the conformal Laplacian, $\hat{u}_0 = u_0/\vp$ solves 
\[ \triangle_{g_0} \hat{u}_0 + c_n \Scal_{g_0} \hat{u}_0 = \hat{u}_0^{2^*-1}\quad\text{in }M,\]
where we have let $g_0:= \vp^{2^*-2}g$. Assume that one of the positive minimizers $u_0$ achieving $Y(M,[g])$ is non-degenerate as a solution of the Yamabe equation. This means that 
\ben \label{u0nondeg}
  \textrm{Ker} \big(  \triangle_{g} + c_n \Scal_{g} - (2^*-1) u_0^{2^*-2} \big) = \{0\},
  \een
where this kernel is regarded as a subset of $H^1(M)$.  A simple application of the Implicit Function Theorem shows that for any metric $g_*$ close enough to $g$ in some $C^p$ topology, $p \ge 3$, there exists a unique positive $u_* \in C^2(M)$ close to $u_0$ in $C^2(M)$ that solves 
\ben \label{ustar}
 \triangle_{g_*} u_* + c_n \Scal_{g_*} u_* = u_*^{2^*-1}\quad\text{in }M 
 \een
and that is also non-degenerate. Generically with respect to perturbations of the metric, at least in dimensions $n\le24$ (see Theorem~10.3 of Khuri--Marques--Schoen \cite{KhuMaSc}), all positive solutions of \eqref{IntroEq4} are non-degenerate. In the locally conformally flat case, concrete examples of situations where $u_0$ is non-degenerate are given by $u_0:=\(\(n-2\)/2\)^{\frac{n-2}{2}}$ on $\S^1\(r\)\times\S^{n-1}$, where $\S^1\(r\)$ is the circle of radius $r\in\(0,\infty\)\backslash\left\{i/\sqrt{n-2}:i\in\N\right\}$ and $\S^{n-1}$ is the unit $(n-1)$-sphere, both equipped with their standard metrics (see Proposition~3.4 of Robert--V\'etois \cite{RobertVetois3}). 

\smallskip
We introduce the following definition:

\begin{definition}  \label{nondege}
Let $(M,g)$ be a smooth, closed Riemannian manifold of dimension $n \ge 3$ and positive Yamabe type. We say that $g$ is $Y$-non-degenerate if 
\begin{itemize}
\item one of the positive minimizers $u_0$ achieving $Y(M,[g])$ is non-degenerate
\item and there exists a constant $\eps_g >0$ such that for any metric $g_*$ satisfying $\Vert g_* - g\Vert_{C^3(M)} \le \eps_g$, the unique function $u_*$ close to $u_0$ in $C^2(M)$ satisfying \eqref{ustar} still attains $Y(M,[g_*])$, i.e. satisfies 
\[ \int_M u_*^{2^*} dv_{g_*} = Y(M,[g_*])^{\frac{n}{2}}. \] 
\end{itemize} 
\end{definition}

By the conformal invariance of $\triangle_g + c_n \Scal_g$, if $g$ is Y-non-degenerate in $M$ then any metric in the conformal class  of $g$ is still Y-non-degenerate. This notion allows us to state a generalization of Theorem \ref{Th0}:
 
\begin{theorem} \label{Th1}
Let $(M,\hat{g})$ be a smooth, closed, locally conformally flat Riemannian manifold of dimension $n \ge 11$ which is \emph{Y-non-degenerate} in the sense of Definition \ref{nondege}. Then there exist a smooth, non-locally conformally flat metric $g$ in $M$ and a sequence of blowing-up sign-changing solutions $(u_k)_{k\in\N}$, $u_k \in C^{3,\alpha}(M)$ for some $0 < \alpha < 1$, to the nodal Yamabe equation for $g$:
\[ \triangle_{g} u_k + \frac{n-2}{4(n-1)} \Scal_{g} u_k = |u_k|^{2^*-2} u_k\quad\text{in }M,\]
which satisfies 
\[ \int_M |u_k|^{2^*}dv_{g} \nearrow Y(M,[g])^{\frac{n}{2}} + Y(\mathbb{S}^n,[g_{std}])^{\frac{n}{2}}   \]
as $k \to \infty$. 
\end{theorem}

The metric $g$ in Theorem \ref{Th1} can be chosen arbitrarily close to the original metric $\hat{g}$ in $C^p(M)$ for any $p \ge 3$. Theorem \ref{Th1} is proven in Sections \ref{secLS} and \ref{secreduced}. In the rest of this section, we prove that spherical space forms and their locally conformally flat perturbations are $Y$-non-degenerate and that Theorem \ref{Th0} follows from Theorem \ref{Th1}. The first result is as follows:

\begin{proposition}\label{YND}
Let $n \ge 3$ and let $\Gamma$ be a finite subgroup of isometries of $(\mathbb{S}^n, g_{std})$, $\Gamma \neq \{Id\}$ acting freely and smoothly on $\mathbb{S}^n$. The quotient manifold $M_\Gamma:= \mathbb{S}^n/\Gamma$ endowed with the quotient metric $g_\Gamma$ is locally conformally flat and Y-non-degenerate in the sense of Definition \ref{nondege}.
\end{proposition}

\begin{proof}
By definition, ($M_\Gamma$,$g_\Gamma$) is locally isometric to $(\mathbb{S}^n, g_{std})$, hence it is locally conformally flat, has constant scalar curvature equal to $n(n-1)$ and is thus of positive Yamabe type. 

The constant function $u_0 \equiv \( n(n-2)/4\)^{\frac{n-2}{4}}$ is a solution of the Yamabe equation on $(\mathbb{S}^n, g_{std})$, 
so $u_0$ descends to $M_\Gamma$ as a constant positive solution, that we still denote by $u_0$, of the Yamabe equation
\ben \label{Mg}
 \triangle_{g_\Gamma} u_0 + \frac{n(n-2)}{4} u_0 = u_0^{2^*-1} \quad \textrm{in } M_\Gamma.
 \een
We claim that the linearized operator at $u_0$ in $M_\Gamma$, which is given by $L = \triangle_{g_{\Gamma}} - n$, has zero kernel. Indeed, if $L \vp = 0$ for some $\vp \in H^1(M_\Gamma)$ then $\vp$ lifts to $\mathbb{S}^{n}$ as a function $\tilde \vp \in H^1(\mathbb{S}^n)$ which is $\Gamma$-invariant and satisfies 
$ \triangle_{g_{std}} \tilde \vp- n \tilde \vp = 0$. But the kernel of $\triangle_{g_{std}} - n $ in $\mathbb{S}^n$ consists of the restrictions of the coordinate functions $(x_i)_{0 \le i \le n}$ of $\R^{n+1}$ to $\mathbb{S}^{n}$, that are not $\Gamma$-invariant. Hence $\vp \equiv 0$ and $u_0$ is a non-degenerate solution of the Yamabe equation in $M_\Gamma$. This proves the first point in Definition \ref{nondege}.

We now prove the second point. Since $(M_\Gamma, g_\Gamma)$ is Einstein, a celebrated theorem of Obata \cite{Obata} shows that $u_0$ is the only positive solution of \eqref{Mg}, so in particular $u_0$ satisfies
\[ \int_{M_\Gamma} u_0^{2^*} dv_{g_{\Gamma}} = Y(M_\Gamma, [g_{\Gamma}])^{\frac{n}{2}} = \frac{1}{|\Gamma|} Y(\mathbb{S}^n, [g_{std}])^{\frac{n}{2}}. \]
By the Implicit Function Theorem there is $\eta_0 >0$ such that, for any $\eps$ small enough and any metric $g$ on $M_\Gamma$ such that $\Vert g - g_{\Gamma} \Vert_{C^3(M_\Gamma)}\le \eps$, there is a unique positive function $u_g \in C^2(M_\Gamma) \cap B_{H^1(M_\Gamma)}(u_0, \eta_0)$ satisfying 
\ben \label{Yg}
 \triangle_g u_g + c_n \Scal_g u_g = u_g^{2^*-1}\quad\text{in }M.
 \een
We also have $\Vert u_g - u_0 \Vert_{C^2(M_\Gamma)} \le C \eps $ for some $C>0$ independent of $\eps$, so $u_g$ is still non-degenerate for $\eps$ small enough. We claim that for $\eps$ small enough, we again have 
\[ \int_{M_\Gamma} u_g^{2^*} dv_g = Y (M_\Gamma,[g])^{\frac{n}{2}} \]
for any $g$ with $\Vert g - g_{\Gamma}\Vert_{C^3(M_\Gamma)} \le \eps$. Assume by contradiction that, for a sequence of real numbers $(\eps_j)_{j\in\N}$ such that $\eps_j \to 0$ as $j\to\infty$ and a sequence of metrics $(g_j)_{j\in\N}$ with $\Vert g_j - g_{\Gamma} \Vert_{C^3(M_\Gamma)} \le \eps_j$, the latter equality does not hold. Then
\[  \int_{M_\Gamma} u_j^{2^*} dv_{g_j} > Y(M_\Gamma, [g_j])^{\frac{n}{2}} \quad\forall j\in\N,\]
where we have let $u_j := u_{g_j}$. Since $g_j \to g_{\Gamma}$ in $C^3(M_\Gamma)$, we have $Y(M_\Gamma, [g_j]) \to Y(M_\Gamma, [g_{\Gamma}]) < Y(\mathbb{S}^n,[g_{std}])$ and thus, by the resolution of the Yamabe problem, there exists a sequence $(v_j)_{j\in\N}$ of positive solutions of \eqref{Yg} with $g = g_j$ satisfying 
\ben \label{Yg2}
 \int_{M_\Gamma} v_j^{2^*} dv_{g_j} = Y(M_\Gamma, [g_j])^{\frac{n}{2}} \le \frac{1}{|\Gamma|}Y(\mathbb{S}^n, [g_{std}])^{\frac{n}{2}} + \smallo\(1\)
 \een
as $j \to  \infty$. In particular, $u_j \neq v_j$ for all $j\in\N$. Struwe's $H^1$-compactness result \cite{Struwe} together with \eqref{Yg2} show that the sequence $(v_j)_{j\in\N}$ strongly converges in $H^1(M_\Gamma)$ as $j \to \infty$ towards a positive solution of \eqref{Mg}. By the uniqueness result of Obata \cite{Obata}, $v_j \to u_0$, and hence the local uniqueness shows that $v_j = u_j$ for large $j$, a contradiction.
\end{proof}

\begin{proof}[Proof of Theorem \ref{Th0}.]
Thanks to Proposition \ref{YND}, we can apply Theorem \ref{Th1} to $(M,\hat{g}) = (M_\Gamma, g_\Gamma)$ for some finite subgroup $\Gamma \neq \{Id\}$ of isometries of $(\mathbb{S}^n, g_{std})$ acting freely and smoothly on $\mathbb{S}^n$. We can then let $g$ be as in the statement of Theorem \ref{Th1}. By definition of $E(M, [g])$, Theorem \ref{Th1} then gives 
\[ E(M, [g]) \le  Y(M,[g])^{\frac{n}{2}} + Y(\mathbb{S}^n,[g_{std}])^{\frac{n}{2}}. \]
The other inequality follows from Proposition \ref{Pr1} below, which concludes the proof of Theorem \ref{Th0}.
\end{proof}

The arguments developed in the proof of Proposition \ref{YND} similarly show that if a metric $g$ in $M$ attains $Y(M,[g])$ at a unique positive minimizer $u_0 $, that is also non-degenerate in the sense of \eqref{u0nondeg}, then $(M,g)$ is $Y$-non-degenerate. With this observation, we can prove the following result  that provides additional $Y$-non-degenerate examples to which Theorem \ref{Th1} applies:

\begin{proposition} \label{propgamma}
 Let $n \ge 3$, let $\Gamma$ be a finite subgroup of isometries of $(\mathbb{S}^n, g_{std})$, $\Gamma \neq \{Id\}$ acting freely and smoothly on $\mathbb{S}^n$ and let $g$ be a locally conformally flat metric in $M_\Gamma$, $\Gamma \neq \{Id\}$, that is close to $g_\Gamma$ in $C^p(M_\Gamma)$ for some $p$ large enough. Then $g$ is $Y$-non-degenerate. 
\end{proposition}

\begin{proof}
This is a consequence of the uniqueness result of de Lima--Piccione--Zedda \cite{LimaPiccioneZedda} (Theorem $5$). We provide some additional details here since the result of \cite{LimaPiccioneZedda} is not stated in this way. The analysis in \cite{LimaPiccioneZedda} leading to Theorem $5$ applies as long as sequences of Yamabe metrics close to $g_\Gamma$ can be made to converge strongly up to a subsequence. This is the case for sequences of locally conformally flat metrics $(g_k)_{k\in\N}$ on $M_\Gamma$ close to $g_\Gamma$ in  $C^p(M_\Gamma)$ for some $p$ large enough. Local arguments indeed show that any metric $g_k$ has positive Riemannian mass, with a positive uniform bound from below, at every point of $M_\Gamma$. This is a purely local argument that works for any dimension $n \ge 3$ and does not rely on the positive mass theorem. The convergence of $(g_k)_{k\in\N}$ up to a subsequence then follows from the arguments of Schoen \cite{Schoenlcf} and Li--Zhu \cite{LiZhu}. Hence  for any metric $g$ on $M_\Gamma$ with $|\Gamma| \ge 2$ that is locally conformally flat and $C^{p}$-close to $g_\Gamma$ for $p$ large enough, the equation \eqref{Yg} has a unique positive solution $u_g$. This solution $u_g$ remains non-degenerate for $g$ close enough to $g_\Gamma$.
\end{proof}

\begin{remark}
A natural question connected with Definition \ref{nondege}, and with the observation before Proposition \ref{propgamma}, is whether it is possible that the Yamabe equation admits multiple minimizers. This is indeed the case. Examples of such situations can be obtained by considering manifolds with a nontrivial isometry group. In particular, Schoen \cite{Schoenlcf} gave a detailed study of the multiplicity of positive solutions to the Yamabe equation in the case of the product manifold $\S^1\(r\)\times\S^{n-1}$ with $r>0$. In this case, if $r$ is chosen large enough, then there exists a family of distinct (degenerate) minimizers parametrized by $\S^1$. This example can be extended to more general manifolds with different isometry groups (see Hebey--Vaugon \cite{HebeyVaugon}).
\end{remark}

\section{Proof of Theorem \ref{Th1} -- Part $1$: A Lyapunov--Schmidt reduction} \label{secLS}

\subsection{The geometric setting}

In this section and the next, we prove Theorem \ref{Th1}. Throughout the paper, we denote by $\delta_0$ the Euclidean metric in $\R^n$. 

\smallskip
Let $(M,\hat{g})$ be a smooth, closed, locally conformally flat Riemannian manifold of positive Yamabe type that is $Y$-non-degenerate in the sense of Definition \ref{nondege}. In this section and the next, we always assume that $n:= \text{dim}(M)\ge 11$. Fix $x_0 \in M$ once and for all, and let $\delta >0$ and $\vp \in C^{\infty}(B_{\hat{g}}(x_0, 8 \delta))$, $\vp>0$, be such that $g_0 = \vp^{2^*-2} \hat{g}$ is flat in $B_{\hat{g}}(x_0, 8 \delta)$. By decreasing $\delta$ if necessary and picking a local chart $\Phi$ that sends $x_0$ to $0$, we can assume that $B_{\hat{g}}(x_0, 6 \delta)$ contains $\Phi^{-1}(B(0, 4\delta))$ and that $\Phi_*g_0$ is the Euclidean metric in $B(0, 4 \delta) \subset \R^n$, where $B(0, 4 \delta)$ is an Euclidean ball. 

\smallskip
For any $k\in\N$, we let $y_k = ( \delta/k, 0, \dotsc, 0) \in \R^n$ and we let $(r_k)_{k\in\N}$  be a decreasing sequence of positive numbers converging to $0$ such that $r_0 \le \delta$ and $4 r_k \le |y_k - y_{k+1}|$ for all $k\in\N$. Any two Euclidean balls $B(y_k, 2r_k)$ and $B(y_\ell, 2 r_\ell)$ are thus disjoint for $k \neq \ell$. Let $(\eps_k)_{k\in\N}$ be a sequence of positive numbers converging to $0$ such that  $\eps_k = o(r_k^p)$ for any $p \ge 1$. Define, for any $k\in\N$, 
\ben \label{defmuk}
 \mu_k = \eps_k^{\frac{4}{n-10}} .\een
Since $n \ge 11$ and $\eps_k = o(r_k^p)$, we have $\mu_k = o(r_k^p)$ as $k \to  \infty$ for any $p \ge 1$. Let $h$ be a smooth, symmetric bilinear form in $\R^n$ that satisfies 
\ben \label{hyph1}
 \textrm{tr}\,h(x)= 0, \quad  \textrm{div}\,h(x)_i = 0 \quad \text{ and }  \sum_{j=1}^n x_j h_{ij} (x)= 0 
 \een 
 for all  $x\in \R^n$ and $1 \le i \le n$, where we have let $\textrm{tr}\,h := \sum_{j=1}^n h_{jj}$ and $\textrm{div}\,h_i := \sum_{j=1}^n \partial_j h_{ji}$. These are respectively the trace and the divergence of $h$ with respect to the Euclidean metric $\delta_0$, but the subscript $\delta_0$ will be omitted for clarity.  We assume that for any $1 \le i,j \le n$, $x \mapsto h_{ij}(x)$ is a homogeneous polynomial of second-order in $\R^n$. 
Examples of such $h$ satisfying \eqref{hyph1} are given in \eqref{defh} below. Let $\chi \in C^{\infty}_c(\R)$ be such that $\chi \equiv 1$ on $[0,1]$ and $\chi \equiv 0 $ on $\R \backslash [0,2]$. We define a new metric in $B(0, 4 \delta)$ by
\ben \label{deftg}
 \tilde g (x):= \exp \( \sum_{k=1}^\infty \eps_k  \chi \(\frac{|x-y_k|}{r_k} \) h\(x-y_k\)  \). 
 \een
We assume in addition that $\sum_{k \in \mathbb{N}} \varepsilon_k r_k^{-p} < + \infty$ for all $p \ge 0$. Since 
the components of $h$ are homogeneous polynomials of second order, $\tilde g $ is thus a smooth metric in $B(0, 4 \delta)$, is flat in $B(0, 4 \delta) \backslash B(0, 3 \delta)$ and satisfies
\ben \label{deftg2}
 \tilde g (x)= \exp \( \eps_k  \chi \(\frac{|x-y_k|}{r_k} \) h\(x-y_k\)  \) 
 \een
in $B(y_k, 2 r_k)$ for any $ k\in\N$. When extended and pulled back to $M$, $\Phi^*\tilde g$ defines a metric in $B_{\hat{g}}(x_0, 8 \delta)$, equal to $g_0$ in $B_{\hat{g}}(x_0, 8 \delta) \backslash B_{\hat{g}}(x_0, 6 \delta)$. The metric $g = \vp^{2 - 2^*}\Phi^* \tilde{g}$ thus defines a smooth metric in $B_{\hat{g}}(x_0, 8 \delta)$, equal to the original metric $\hat{g}$ in $B_{\hat{g}}(x_0, 8 \delta) \backslash B_{\hat{g}}(x_0, 6 \delta)$, that we extend to be equal to $\hat{g}$ on $M\backslash B_{\hat{g}}(x_0, 8 \delta) $. We still call this new metric $g$. We now define $\check{g} = \vp^{2^*-2} g$, which is a smooth metric in $M$ such that $\Phi_* \check{g} = \tilde g$ as in \eqref{deftg} in $B(0, 4 \delta)$.  By \eqref{hyph1}, we have $\det \tilde g \equiv 1 $ in $B(0, 4 \delta)$. Note that $\tilde g$ can be chosen to be arbitrarily close to the Euclidean metric $\delta_0$ in $C^p(B(0,4\delta))$ for any $p \ge 3$ by assuming that  $\sum_{k \in \mathbb{N}} \varepsilon_k r_k^{-p} $ is small enough. Hence $g$ can be chosen arbitrarily close to the initial metric $\hat{g}$ in $C^p(M)$ for any $p \ge 3$.

\smallskip
Since $\hat{g}$ is $Y$-non-degenerate, so is $g_0 = \vp^{2^*-2} \hat{g}$, and we can let $\hat{u}_0$ be a non-degenerate positive solution of 
\[ \triangle_{g_0} \hat{u}_0 + c_n \Scal_{g_0} \hat{u}_0 = \hat{u}_0^{2^*-1}\quad\text{in }M,\]
that also satisfies 
$$ \int_M \hat{u}_0^{2^*} dv_{g_0} = Y(M,[g_0])^{\frac{n}{2}} =  Y(M,[\hat{g}])^{\frac{n}{2}}.$$
Definition \ref{nondege} then yields the existence of a unique positive function $\check{u}_0 \in C^2(M)$ solving 
\ben \label{deftu0}
 \triangle_{\check g} \check{u}_0 + c_n \Scal_{\check g} \check{u}_0 = \check{u}_0^{2^*-1}\quad\text{in }M,
 \een
that is still non-degenerate in the sense of \eqref{u0nondeg} and satisfies 
\[\int_M \check{u}_0^{2^*} dv_{\check g} = Y(M,[\check{g}])^{\frac{n}{2}} = Y(M,[g])^{\frac{n}{2}}.\] 
In the rest of this section and in the following one, we construct, when $n \ge 11$, a sequence of sign-changing solutions $(u_k)_{k\in\N}$ of class $C^{3,\alpha}(M)$, $0 < \alpha \le 2^*-2$, to
\ben \label{ynodal}
 \triangle_{\check g} u_k + c_n \Scal_{\check g} u_k = |u_k|^{2^*-2} u_k\quad\text{in }M,\een
where $c_n := \frac{n-2}{4(n-1)}$, that satisfies
\[ \int_M |u_k|^{2^*} dv_{\check{g}}  \nearrow Y(M,[\check g])^{\frac{n}{2}} + Y(\mathbb{S}^n,[g_{std}])^{\frac{n}{2}} . \]
Since $\check{g} = \vp^{2^*-2} g$, and by the conformal invariance of the conformal Laplacian, by replacing $u_k$ with  $ \vp u_k$, this will prove Theorem \ref{Th1}.

\subsection{The \emph{ansatz} of the construction}

Define $D^{1,2}(\R^n)$ to be the completion of $C^\infty_c(\R^n)$ for the norm $u \mapsto \Vert \nabla u \Vert_{L^2(\R^n)}$, that we endow with the associated scalar product. We fix $A >0$ to be chosen later, and for $(t,z)\in [1/A,A] \times B(0,1)$, we let  
\ben \label{defsuites}
\mu_k(t) := \mu_k t\quad\text{and} \quad  \xi_k(z) := y_k + \mu_k z,
\een
where $y_k$ is as in the previous subsection, and for $x \in \R^n$,
\[ B_{k,t,z}(x) = \frac{\mu_k (t)^{\pui}}{\(\mu_k( t)^2 + \frac{|x-\xi_k(z)|^2}{n(n-2)}\)^{\pui}}. \] 
For any $(t,z) \in  [1/A,A]\times B(0,1)$, $B_{k,t,z}$ solves 
\[ \triangle_{\delta_0} B_{k,t,z} = B_{k,t,z}^{2^*-1}\quad\text{in }\R^n, \]
where $\delta_0$ is the Euclidean metric in $\R^n$ and $\triangle_{\delta_0} := - \textrm{div}_{\delta_0}\nabla$. For $x \in \R^n$, we let 
\ben \label{defVi}
V_0(x) := \frac{\frac{|x|^2}{n(n-2)} - 1}{\(1 + \frac{|x|^2}{n(n-2)}\)^{\frac{n}{2}}} \quad  \textrm{ and, for  }  1 \le j \le n, \quad  V_j(x): = \frac{x_j}{\(1 + \frac{|x|^2}{n(n-2)}\)^{\frac{n}{2}}} .
\een
We then let 
\ben \label{Z0}
 Z_{0,k,t,z}(x) ;= \mu_k(t)^{1 - \frac{n}{2}} V_0 \(\frac{x - \xi_k(z)}{\mu_k(t)}  \) =   \frac{2}{n-2} t \partial_t B_{k,t,z} \een
and, for $1 \le j \le n$, 
\ben \label{Zi}
 Z_{j,k,t,z}(x); =  \mu_k(t)^{1 - \frac{n}{2}} V_j \(\frac{x - \xi_k(z)}{\mu_k(t)}  \) =  - n t \partial_{z_j} B_{k,t,z} 
 ,\een
and we let 
\[ K_{k,t,z} ;= \textrm{span} \left\{ Z_{j,k,t,z}:\, 0 \le j \le n \right\},\]
which is a finite-dimensional subspace of $D^{1,2}(\R^n)$. We denote by $K_{k,t,z}^{\perp}$ its orthogonal complement in $D^{1,2}(\R^n)$. The functions $Z_{i,k,t,z}$ satisfy
\[ \triangle_{\delta_0} Z_{i,k,t,z} = (2^*-1) B_{k,t,z}^{2^*-2} Z_{i,k,t,z}\quad\text{in }\R^n \]
for all $0 \le j \le n$, and by a result of Rey \cite{Rey} and Bianchi-Egnell \cite{BianchiEgnell}, they form an orthogonal basis of the set of solutions of this equation in $D^{1,2}(\R^n)$. Letting $h$ be as in the previous subsection, we define, for $k\in\N$ and $x \in \R^n$,
\ben \label{defhk}
 h_k(x) := h \(x-y_k\).
 \een
 By \eqref{hyph1}, $h_k$ is trace-free and divergence-free in $\R^n$. As a first result, we obtain the following lemma:

\begin{lemma} \label{lemmeR}
For any $(t,z) \in [1/A, A] \times B(0,1)$ and $k\in\N$, there exists a unique $R_{k,t,z} \in K_{k,t,z}^{\perp}$ that satisfies
\ben \label{eqRk}
 \triangle_{\delta_0} R_{k,t,z} - (2^*-1) B_{k,t,z}^{2^*-2} R_{k,t,z} = - \eps_k \sum_{p,q=1}^n (h_k)_{pq} \partial_{pq}^2 B_{k,t,z} \quad\text{in }\R^n.
 \een
This function $R_{k,t,z}$ also satisfies, for $i\in\left\{0,1,2\right\}$,
\begin{equation}\label{estR}
|\nabla^i R_{k,t,z}(x)| \le C \eps_k \mu_k(t)^{\frac{n+2}{2}} \frac{\ln \( \frac{2 \mu_k(t) + |x - \xi_k(z)|}{\mu_k(t)} \)}{ \(\mu_k(t) + |x - \xi_k(z)| \)^{n-2+i}} 
\end{equation}
for all $x \in \R^n$, for some $C >0$ independent of $k,t,z$.
\end{lemma}

\begin{proof}
Let $u \in C^2(\R^n)$ be such that $u \in L^2(\R^n)$, $(1+|x|)|\nabla u(x)| \in L^2(\R^n)$ and $(1+|x|)^2|\nabla^2 u(x)| \in L^2(\R^n)$.  Define then
\[ G_k(u) := \frac12 \int_{\R^n} \sum_{p,q=1}^n (h_k)_{pq} \partial_p u\,\partial_q u\,dx . \]
Integrating by parts and since $h_k$ is divergence free, we get that 
\[ G_k(u) = - \frac12 \int_{\R^n} \sum_{p,q=1}^n (h_k)_{pq} u\,\partial^2_{pq}u \,dx. \]
Choose now $u = B_{k,t,z}$ for $t >0$ and $z \in B(0,1)$, which is admissible for $G_k$ since $n \ge 11$. Explicit computations show that 
\[ B_{k,t,z} \,\partial_{pq}^2 B_{k,t,z}  = - \frac{1}{n} \frac{\mu_k(t)^{n-2} \delta_{pq}}{\( \mu_k(t)^2 + \frac{|x - \xi_k(z)|^2}{n(n-2)} \)^{n-1}} + \frac{n}{n-2} \,\partial_p B_{k,t,z} \,\partial_q B_{k,t,z}, \]
which gives, since $h_k$ is trace-free, $G_k(B_{k,t,z}) = - \frac{n}{n-2} G_k(B_{k,t,z})$, and so $G_k(B_{k,t,z}) = 0$. Differentiating with respect to $t$ and $z$, using \eqref{Z0} and \eqref{Zi} and integrating by parts shows that for any $j \in \{0, \dotsc, n\}$, 
\[ \int_{\R^n} Z_{j,k,t,z} \sum_{p,q=1}^n (h_k)_{pq} \,\partial_{pq}^2 B_{k,t,z} \,dx = 0 \]
holds. Since, by Bianchi--Egnell \cite{BianchiEgnell}, $\triangle_{\delta_0} - (2^*-1) B_{k,t,z}^{2^*-2}$ is Fredholm and injective on $K_{k,t,z}^{\perp}$ the existence of a unique $R_{k,t,z} \in K_{k,t,z}^{\perp}$ satisfying \eqref{eqRk} follows. By standard elliptic theory $R_{k,t,z}$ is smooth.  

By \eqref{hyph1}, $h_k$ is trace-free and satisfies $\sum_{j=1}^n (x-y_k)_j (h_k)_{ij}(x) = 0$ for all $1 \le j \le n$ and $x \in \R^n$. Direct computations then show that 
\[ \sum_{p,q=1}^n (h_k)_{pq} \,\partial_{pq}^2 B_{k,t,z} = - \frac{\mu_k^2}{n(n-2)} \sum_{p,q=1}^n(h_k)_{pq}z_p z_q \frac{ \mu_k(t)^{\pui}}{\(\mu_k(t)^2 + \frac{|\cdot-\xi_k(z)|^2}{n(n-2)} \)^{\frac{n+2}{2}}}. \]
It is then easily seen that, for all $x \in \R^n$, 
\ben \label{Rktz}
R_{k,t,z}(x) = \eps_k \mu_k^{3 - \frac{n}{2}} R_{t,z} \( \frac{x-y_k}{\mu_k}\)
\een
holds, where $R_{t,z}$ is the unique solution in $K_{t,z}^{\perp}$ of 
\ben \label{Rtz}
\triangle_{\delta_0} R_{t,z} - (2^*-1) B_{t,z}^{2^*-2} R_{t,z} = - \frac{t^{\pui}}{n(n-2)}  \frac{ \sum_{p,q=1}^n h_{pq}z_p z_q}{\(t^2 + \frac{|\cdot-z|^2}{n(n-2)} \)^{\frac{n+2}{2}}}\quad\text{in }\R^n,  
\een
where we have let 
\ben \label{Btz}
B_{t,z}(x) := t^{\pui} \(t^2 + \frac{|x-z|^2}{n(n-2)} \)^{- \pui}
\een
 and 
\[ K_{t,z} := \textrm{span} \left\{  \partial_t B_{t,z},\partial_{z_1} B_{t,z},\dotsc,\partial_{z_n} B_{t,z}\right\} . \]
Since $t \in [1/A,A]$ and $z \in B(0,1)$ and since the $h_{ij}$ are homogeneous of degree $2$, we can again write 
$$R_{t,z}(x) = t^{1-\frac{n}{2}}S_{0,z}\(\frac{x-z}{t}\)$$ 
for some smooth function $S_{0,z} \in K_{1,0}^{\perp}$ that satisfies
\[ \big| \triangle_{\delta_0} S_{0,z}(y) - (2^*-1) B_{1,0}(y)^{2^*-2} S_{0,z}(y)\big| \le C |z|^2 \(1+ |y|\)^{-n}  \]
for all $y \in \R^n$, for some $C>0$ independent of $t$ and $z$. We can now write a representation formula for $\triangle_{\delta_0} - (2^*-1)B_{1,0}^{2^*-2} $ (see for example Lemma 3.3 of Premoselli \cite{Premoselli12}) that shows that, for any $x \in \R^n$,
\[ |S_{0,z}(x)| \le C |z|^2 \int_{\R^n} |x-y|^{2-n} (1+ |y|)^{-n} dy \le C|z|^2 \frac{\ln(2+|x|)}{\(1+|x|\)^{n-2}} \]
for some constant $C>0$ independent of $z$. Differentiating the representation formula similarly yields
\[  |\nabla S_{0,z}(x)| \le C|z|^2 \frac{\ln(2+|x|)}{\(1+|x|\)^{n-1}}\quad\text{and}\quad |\nabla S_{0,z}(x)|^2 \le C|z|^2 \frac{\ln(2+|x|)}{\(1+|x|\)^{n}}\]
for any $x \in \R^n$. Going back to $R_{k,t,z}$, this proves \eqref{estR}.
\end{proof}

\smallskip
 Let again $\chi \in C^{\infty}_c(\R)$ be such that $\chi \equiv 1$ on $[0,1]$, $0 \le \chi \le 1$  and $\chi \equiv 0 $ on $\R \backslash [0,2]$. For $t \in [1/A,A] $ and $z \in B(0,1)$, we let 
 \ben \label{defU}
U_{k,t,z}(x): = \chi \( \frac{x-y_k}{r_k} \) \( B_{k,t,z}(x) + R_{k,t,z}(x) \)
\een
for any $x \in \R^n$, and 
 \ben \label{defW}
W_{k,t,z}(x): =  U_{k,t,z}\(\Phi(x) \)
- \check{u}_0(x)
\een
for any $x \in M$, where $\Phi$ is the chart around $x_0$ introduced in the previous subsection, and where $\check{u}_0$ is as in \eqref{deftu0}. 
We let 
\ben \label{deff}
f(s): = |s|^{2^*-2}s \quad \forall s \in \R.
\een We also let for $k\in\N$ and $(t,z) \in [1/A,A] \times B(0,1)$,
\ben \label{erreur}
 E_{k,t,z} := \(  \triangle_{\check g} + c_n \Scal_{\check g} \)W_{k,t,z} - f(W_{k,t,z}). \een
 Until the end of this section, $C$ will denote a positive constant independent of $k,t$ and $z$, that might change from one line to the other. 
 
 \begin{lemma}
For any $k\in\N$ and $(t,z) \in [1/A,A] \times B(0,1)$, we have
\ben \label{esterreur}
 \Vert E_{k,t,z} \Vert_{L^{\frac{2n}{n+2}}(M)} \le C \mu_k^{\frac{n+2}{4}} .
 \een
\end{lemma}

\begin{proof}
By \eqref{defmuk} and \eqref{estR}, we have $\left\| R_{k,t,z}/B_{k,t,z} \right\|_{L^\infty(B(y_k, 2r_k))} \to 0$ as $k \to  \infty$ uniformly with respect to $t,z$. As a consequence, and since $\mu_k^2 = \smallo\(r_k\)$,
\[
 \big| f(W_{k,t,z}) - f(U_{k,t,z}\circ\Phi) + f(\check{u}_0) \big| \le C \left \{ \bal & \big(B_{k,t,z} \circ \Phi \big)^{2^*-2} && \textrm{in } \Phi \big(B(\xi_k(z), \sqrt{\mu_k(t)})\big) \\ & \big(B_{k,t,z} \circ \Phi \big) && \textrm{otherwise.}\eal \right.  \]
Since $\check{u}_0$ satisfies \eqref{deftu0} and $\tilde g = \Phi_* \check{g}$ in $B(0, 4 \delta)$, we obtain
\ben \label{err1} \Vert E_{k,t,z} \Vert_{L^{\frac{2n}{n+2}}(M)} \le \Vert  (  \triangle_{\tilde g} + c_n \Scal_{\tilde g} )U_{k,t,z} - f(U_{k,t,z}) \Vert_{L^{\frac{2n}{n+2}}(B(y_k, 2 r_k))} + C \mu_k^{\frac{n+2}{4}} ,\een
 where $\tilde g$ is given by \eqref{deftg}. First, by using \eqref{estR} together with straightforward computations, we obtain 
\ben \label{err2}
 ( \triangle_{\tilde g} + c_n \Scal_{\tilde g} )U_{k,t,z} - f(U_{k,t,z})  = \bigO \( \mu_k^{\pui} r_k^{-n}\) 
 \een
in $B(y_k, 2r_k) \backslash B(y_k, r_k)$. In $B(y_k, r_k)$, by \eqref{deftg2}, we have 
\[ \tilde g = \exp \( \eps_k h_k(x) \),  \]
where $h_k$ is defined in \eqref{defhk}. Since the components of $h$ are homogeneous of degree two, $ |\eps_k h_k(x)|  \le C \eps_k r_k^2 = \smallo\(1\)$ holds for all $x \in B(y_k, r_k)$. The definition of $\tilde g$ therefore allows to expand its inverse as 
\ben \label{calc1}
\tilde g^{ij}(x) = \delta_{ij} - \eps_k (h_k)_{ij}(x) + \frac{\eps_k^2}{2} \sum_{p=1}^n (h_k)_{ip}(x)(h_k)_{pj}(x) + \bigO\(\eps_k^3 |x-y_k|^6 \)
\een
for $i,j \in \{1, \dotsc, n\}$. Similarly, the Christoffel symbols of $\tilde g$ expand as 
\begin{equation} \label{calc2} 
\Gamma_{ij}^\ell (\tilde g)(x)  = \frac{\eps_k}{2} \(\partial_i (h_k)_{j \ell}(x) + \partial_j (h_k)_{i \ell}(x) - \partial_{\ell}(h_k)_{ij}(x) \) + \bigO \( \eps_k^2 |x-y_k|^3 \)
\end{equation}
for $i,j, \ell \in \{1, \dotsc, n\}$. Using Proposition $26$ of Brendle \cite{Brendle}, and since $h_k$ is trace-free and divergence-free, the scalar curvature of $\tilde g$ expands as
\begin{equation} \label{calc3} 
\Scal_{\tilde g} (x)  = - \frac14\eps_k^2  \sum_{i,j,\ell=1}^n (\partial_i (h_k)_{j\ell}(x))^2 + \bigO\( \eps_k^3 |x-y_k|^4 \).
\end{equation}
Remark finally that, by definition of $\xi_k(z)$ in \eqref{defsuites}, there exists $C >1$ such that for any $x \in B(y_k, r_k)$, 
\[ \frac{1}{C} \le \frac{\mu_k + |x-y_k|}{\mu_k + |x - \xi_k(z)|} \le C  \]
holds true. Using \eqref{calc1}, \eqref{calc2} and \eqref{calc3}, we thus have, for $x \in B(y_k, r_k)$,
\begin{align} \label{err3} 
 &( \triangle_{\tilde g} + c_n \Scal_{\tilde g} )B_{k,t,z} - f(B_{k,t,z})\nonumber\\
& \qquad= (\triangle_{\tilde g} - \triangle_{\delta_0})B_{k,t,z} +\bigO \(\eps_k^2 \mu_k^{\pui} (\mu_k + |x-y_k|)^{4 - n } \)\nonumber\\
&\qquad= \eps_k \sum_{p,q = 1}^n (h_k)_{pq} \,\partial^2_{pq} B_{k,t,z} + \bigO \(\eps_k^2 \mu_k^{\pui} (\mu_k + |x-y_k|)^{4 - n } \). 
\end{align}
Since $U_{k,t,z}(x) = B_{k,t,z}(x) + R_{k,t,z}(x)$, by \eqref{estR},
\[  
 \left| f(U_{k,t,z}(x))  - f(B_{k,t,z}(x)) - f'(B_{k,t,z}(x)) R_{k,t,z}(x) \right| \le C B_{k,t,z}(x)^{2^*-3} R_{k,t,z}(x)^{2} 
\]
holds for any $x \in B(y_k,r_k)$. 
With \eqref{defmuk}, \eqref{estR},  \eqref{calc1}, \eqref{calc2}, \eqref{calc3} and \eqref{err3}, we obtain that, in $B(y_k, r_k)$,
\begin{align*} & ( \triangle_{\tilde g} + c_n \Scal_{\tilde g} )U_{k,t,z}  - f(U_{k,t,z}) \\
&\qquad = ( \triangle_{\tilde g} + c_n \Scal_{\tilde g} )B_{k,t,z} - f(B_{k,t,z})  + \triangle_{\delta_0} R_{k,t,z} - f'(B_{k,t,z}(x)) R_{k,t,z} \\
&\qquad\quad+ (\triangle_{\tilde g}  - \triangle_{\delta_0}) R_{k,t,z} + \bigO \(\eps_k^2 (\mu_k + |x - y_k|)^2 |R_{k,t,z}|\) + \bigO \(B_{k,t,z}^{2^*-3} R_{k,t,z}^{2} \) \\
&\qquad =  \bigO \(\eps_k^2 \mu_k^{\pui} (\mu_k + |x-y_k|)^{4 - n } \) 
\end{align*}
holds. With \eqref{defmuk}, \eqref{err1} and \eqref{err2}, this finally shows that
\[ 
\Vert E_{k,t,z} \Vert_{L^{\frac{2n}{n+2}}(M)} \le C\(\mu_k^{\frac{n+2}{4}} + \mu_k^{\pui} r_k^{1-\frac{n}{2}} +\eps_k^2\mu_k^4\)\le C \mu_k^{\frac{n+2}{4}},  
\]
where the last inequality follows since $n \ge 11$ and $\mu_k = \smallo\(r_k^p\)$ for any $ p\ge 1$. 
\end{proof}

\subsection{The Lyapunov--Schmidt reduction}

We endow $H^1(M)$ with the norm 
$$\Vert u \Vert_{H^1(M)} :=\sqrt{\int_M\(|\nabla u|_{\check g}^2 + c_n \Scal_{\check{g}} u^2\)dv_{\check g}},$$ 
and for $u\in H^1(M)$, we let 
\[ I(u) := \frac12 \int_M \( |\nabla u|_{\check g}^2 + c_n \Scal_{\check g} u^2 \)dv_{\check g} - \frac{1}{2^*} \int_M |u|^{2^*} dv_{\check g}. \]
For $(t,z) \in [1/A,A] \times B(0,1)$, $1 \le j \le n$ and $x \in M$, we let
\ben \label{defhatZ}
\hat{Z}_{0,k,t,z}(x): = \frac{2}{n-2}\,t \,\partial_t W_{k,t,z}(x) \quad \textrm{ and } \quad  \hat{Z}_{j,k,t,z}(x) :=  - n\,t\,\partial_{z_j} W_{k,t,z}(x),
\een
and we let 
\[ \hat{K}_{k,t,z} := \textrm{span} \big\{ \hat{Z}_{j,k,t,z}:\,0 \le j \le n \big\}, \]
which is regarded as a subset of $H^1(M)$. We denote by $\hat{K}_{k,t,z}^{\perp}$ its orthogonal complement in $H^1(M)$. The following proposition shows the existence of a canonical solution of \eqref{ynodal} in $\hat{K}_{k,t,z}^{\perp}$:

\begin{proposition} \label{LS}
There exists $C>0$ and, for large $k\in\N$, a function $\vp_k: [1/A, A] \times B(0,1) \to H^1(M)$ of class $C^1$ which is the only solution of 
\[ \Pi_{\hat{K}_{k,t,z}^{\perp}}\( u_{k,t,z} - \big(\triangle_{\check g} + c_n \Scal_{\check g}\big)^{-1}\big( f(u_{k,t,z}) \big) \) = 0 \]
in the set 
$$\left\{ \vp \in \hat{K}_{k,t,z}^{\perp}:\, \Vert \vp \Vert_{H^1(M)} \le C \Vert E_{k,t,z} \Vert_{L^{\frac{2n}{n+2}}(M)} \right\},$$ where we have let $u_{k,t,z} := W_{k,t,z} + \vp_k(t,z)$. In particular, 
\ben \label{errvp}
\Vert \vp_k(t,z) \Vert_{H^1(M)} \le C \Vert E_{k,t,z} \Vert_{L^{\frac{2n}{n+2}}(M)}
\een
for some $C >0$ independent of $k,t,z$. In addition, for large $k\in\N$, the function $u_{k,t,z}$ is a critical point of $I$ (hence a solution of \eqref{ynodal}) if and only if $(t,z)$ is a critical point of the mapping $(t,z) \mapsto I \(u_{k,t,z} \)$. 
\end{proposition}

\begin{proof}
By \eqref{defW}, we have, for $0 \le j \le n$ and $x \in B(y_k, 2r_k)$,
\ben \label{althZ}
 \hat{Z}_{j,k,t,z}\( \Phi^{-1}(x) \) =  \chi \( \frac{x-y_k}{r_k} \)  \( Z_{j,k,t,z}(x)  + Q_{j,k,t,z}(x) \),
 \een
where we have let $Q_{0,k,t,z} :=  \frac{2}{n-2} t \,\partial_t R_{k,t,z}$ and $Q_{j,k,t,z} :=  - n t\,\partial_{z_j} R_{k,t,z}$. By differentiating \eqref{eqRk} and using \eqref{estR}, one gets that, for $0 \le j \le n$, $k\in\N$, $x \in B(y_k, 2r_k)$ and any $(t,z) \in [1/A,A] \times B(0,1)$,  $Q_{j,k,t,z}$ satifies
\ben \label{lapQ}
 \big| \triangle_{\delta_0} Q_{j,k,t,z} - (2^*-1) B_{k,t,z}^{2^*-2} Q_{j,k,t,z}  \big| \le C \eps_k \mu_k^{\frac{n+2}{2}} \( \mu_k + |\xi_k(z) - \cdot| \)^{-n}. \een
As in the proof of \eqref{estR}, a representation formula yields once again that, for $i\in\left\{0,1,2\right\}$,
\ben \label{estQ}
|\nabla^i Q_{j,k,t,z}(x)|\le C \eps_k \mu_k(t)^{\frac{n+2}{2}} \frac{\ln \( \frac{2 \mu_k(t) + |x - \xi_k(z)|}{\mu_k(t)} \)}{ \(\mu_k(t) + |x - \xi_k(z)| \)^{n-2+i}}
 \een
 for $x \in \R^n$. With \eqref{Z0}, \eqref{Zi} and \eqref{defhatZ}, we then obtain, for $0 \le j \le n$ and $x \in B(y_k,2r_k)$, that 
\begin{align} 
\hat{Z}_{j,k,t,z}\( \Phi^{-1}(x) \) & =  \chi \( \frac{x-y_k}{r_k} \)Z_{j,k,t,z}(x) \nonumber\\
  &\quad+ \bigO\( \eps_k \mu_k(t)^{\frac{n+2}{2}} \frac{\ln \( \frac{2 \mu_k(t) + |x - \xi_k(z)|}{\mu_k(t)} \)}{ \(\mu_k(t) + |x - \xi_k(z)| \)^{n-2}} \), \allowdisplaybreaks\label{DLhatZ1}\\
 \nabla \big(\hat{Z}_{j,k,t,z}\circ\Phi^{-1}\big)(x) & = \nabla \(  \chi \( \frac{\cdot-y_k}{r_k} \)Z_{j,k,t,z} \)(x) \nonumber\\
 &\quad+ \bigO\( \eps_k \mu_k(t)^{\frac{n+2}{2}} \frac{\ln \( \frac{2 \mu_k(t) + |x - \xi_k(z)|}{\mu_k(t)} \)}{ \(\mu_k(t) + |x - \xi_k(z)| \)^{n-1}} \)  .\label{DLhatZ2}
\end{align}
With these estimates, we can easily adapt the proof of Proposition 5.1 of Robert--V\'etois \cite{RobertVetois} (see also Esposito--Pistoia--V\'etois \cite{EspositoPistoiaVetois}), and Proposition \ref{LS} follows.
\end{proof}

We let $K_n$ be the optimal constant of the Sobolev embedding $D^{1,2}(\R^n) \hookrightarrow L^{2^*}(\R^n)$. As explained in the introduction, $K_n^{-2} = Y(\mathbb{S}^n,[g_{std}])$. Since the latter is attained by the stereographic projection of the bubbles  $B_{t,z}$ defined in \eqref{Btz}, for all $(t,z) \in [1/A,A] \times B(0,1)$, we have 
\[ \int_{\R^n} B_{t,z}^{2^*} \,dx = \int_{\R^n} |\nabla B_{t,z}|^2 dx =  K_n^{-n}.\]
 The explicit value of $K_n$ is known (see Aubin \cite{Aubin1} and Talenti \cite{Talenti}). The next result is an expansion of $(t,z) \mapsto I\( W_{k,t,z} \)$ as $k \to \infty$. 

\begin{proposition} \label{DLenergie2}
We have, as $k \to  \infty$:
\begin{multline}\label{DLI}
   I\(W_{k,t,z} \) = \frac{1}{n} K_n^{-n} + I(\check{u}_0)+ \mu_k^{\pui} \Bigg (  \Lambda(n) \check{u}_0(x_0) t^{\pui}-\frac12  \int_{\R^n} |\nabla R_{t,z}|^2 dx\\+\frac14  \int_{\R^n}\hspace{-2pt}\sum_{i,j,p=1}^n\hspace{-2pt}h_{ip}\,h_{pj}    \,\partial_i B_{t,z} \,\partial_j B_{t,z} \,dx- \frac{n-2}{32(n-1)}  \hspace{-2pt}\int_{\R^n}\hspace{-2pt} \sum_{i,j,\ell=1}^n (\partial_{i} h_{j \ell})^2 B_{t,z}^2 \,dx+ \smallo\(1\) \Bigg), 
\end{multline}
where $\check{u}_0$ is as in \eqref{deftu0}, $R_{t,z}$ is as in \eqref{Rktz} and where $\Lambda(n)$ is a positive dimensional constant given by \eqref{defLa} below. This expansion holds true in $C^1([1/A,A] \times B(0,1))$.
\end{proposition}

\begin{proof}
First, by \eqref{estR} and \eqref{defU}, it is easily seen that $0 \le U_{k,t,z} \le C B_{k,t,z}$ in $B(y_k, 2r_k)$. By \eqref{defW}, we can then write, for some $0 < \theta < 2^*-2$, that
\[ \bal 
\Big| |W_{k,t,z}|^{2^*} - \(U_{k,t,z} \circ \Phi \)^{2^*} - \check{u}_0^{2^*} + 2^* \( U_{k,t,z} \circ \Phi \)^{2^*-1} \check{u}_0 + 2^* \(U_{k,t,z} \circ\Phi \)\check{u}_0^{2^*-1} \Big| \\
\le C \(\(B_{k,t,z} \circ \Phi \)^{2^*-1-\theta} + \( B_{k,t,z} \circ \Phi \)^{1+\theta}\) \eal \]
holds in $M$. As a consequence, and since $\check{u}_0$ solves \eqref{deftu0}, straightforward computations show that, for $0 <\theta < \frac{2}{n-2}$,
\begin{align} \label{ener00} 
&I \(W_{k,t,z} \)  = I\(U_{k,t,z} \circ \Phi \)  + I(\check{u}_0) + \int_M \( U_{k,t,z} \circ \Phi  \)^{2^*-1} \check{u}_0 dv_{\check g} \nonumber\\
&\qquad+ \bigO \( \int_M \( \(B_{k,t,z} \circ \Phi\)^{2^*-1-\theta} + \( B_{k,t,z} \circ \Phi\)^{1+\theta} \)dv_{\check g}\) \nonumber\allowdisplaybreaks\\
&\quad=  I\(U_{k,t,z} \circ \Phi \)  + I(\check{u}_0) + \int_M \( U_{k,t,z} \circ  \Phi \)^{2^*-1} \check{u}_0\,dv_{\check g} + \bigO\(\mu_k^{\pui(1+\theta)}\). 
\end{align}
Independently, by \eqref{estR} and \eqref{defU}, and since $\check{u}_0$ is of class $C^2$, we have 
\ben \label{intUPhi}
\int_M \(U_{k,t,z} \circ \Phi \)^{2^*-1} \check{u}_0 \,dv_{\check g} = \Lambda(n) \check{u}_0(x_0) \mu_k(t)^{\pui} + \smallo\(\mu_k^{\pui}\), 
\een
where $\mu_k(t) := \mu_k t$ is as in \eqref{defsuites} and where we have let 
\ben \label{defLa}
\Lambda(n) := \int_{\R^n} \( 1+ \frac{|x|^2}{n(n-2)} \)^{- \frac{n+2}{2}} dx. 
\een
We have thus proven that 
\ben \label{ener1}
I \(W_{k,t,z} \) = I\(U_{k,t,z} \circ\Phi \)  + I(\check{u}_0) + \Lambda(n) \check{u}_0(x_0) \mu_k(t)^{\pui} + \smallo\(\mu_k^{\pui}\). 
\een
It remains to expand $I\(U_{k,t,z}\circ \Phi \)$. By definition, we have $\check{g} = \Phi^* \tilde{g}$ in $B_g(x_0, 6 \delta)$, where $\tilde g$ is given by \eqref{deftg}, so by \eqref{defU} and since $\textrm{det}\,\tilde g = 1$, we have
\begin{multline} \label{ener0} 
I\(U_{k,t,z}\circ \Phi \)  
  =  \frac12 \int_{\R^n}\(\left|\nabla \( B_{k,t,z} + R_{k,t,z} \) \right|_{\tilde g}^2 + c_n \Scal_{\tilde g} \( B_{k,t,z} + R_{k,t,z} \)^2\)dx \\
- \frac{1}{2^*} \int_{\R^n} \left| B_{k,t,z} + R_{k,t,z} \right|^{2^*} dx + \smallo\(\mu_k^{\pui}\), 
\end{multline}
 where $R_{k,t,z}$ is given by Lemma \ref{lemmeR}. In the latter equality, we implicitly assumed, in accordance with \eqref{deftg}, that $\tilde g$ has been extended as a metric in $\R^n$ that coincides with the Euclidean metric in $\R^n\backslash B(0,4 \delta)$. First, by \eqref{estR} and \eqref{calc1}, we have
\begin{multline} \label{ener2}
  \frac12 \int_{\R^n} \left|\nabla \( B_{k,t,z} + R_{k,t,z} \) \right|_{\tilde g}^2 \,dx = \frac12 \int_{\R^n} |\nabla B_{k,t,z}|^2 dx+ \int_{\R^n} \langle \nabla B_{k,t,z}, \nabla R_{k,t,z} \rangle \,dx \\
   + \frac12 \int_{\R^n} |\nabla R_{k,t,z}|^2 dx - \frac{\eps_k}{2}\int_{\R^n} \sum_{i,j=1}^n (h_k)_{ij} \,\partial_i B_{k,t,z} \,\partial_j B_{k,t,z} \,dx\allowdisplaybreaks\\
  + \frac{\eps_k^2}{4} \int_{\R^n} \sum_{i,j,p=1}^n (h_k)_{ip}(h_k)_{pj}    \,\partial_i B_{k,t,z} \,\partial_j B_{k,t,z} \,dx\\
  - \eps_k \int_{\R^n}\sum_{i,j=1}^n (h_k)_{ij} \,\partial_i B_{k,t,z} \,\partial_j R_{k,t,z} \,dx + \smallo\(\mu_k^{\pui}\) ,
\end{multline}
 where $h_k$ is as in \eqref{defhk}. Similarly, using \eqref{defmuk}, \eqref{estR}  and \eqref{calc3}, we have 
\begin{multline} \label{ener3}
 \frac{c_n}{2}\int_{\R^n}  \Scal_{\tilde g} \( B_{k,t,z} + R_{k,t,z} \)^2 dx = - \frac{c_n}{8} \eps_k^2 \int_{\R^n} \sum_{i,j,\ell=1}^n (\partial_{i} (h_k)_{j \ell})^2 B_{k,t,z}^2 \,dx\\+ \smallo \(\mu_k^{\pui}\).
\end{multline}
 Finally, using \eqref{estR}, we have
\begin{multline} \label{ener3bis}
\frac{1}{2^*} \int_{\R^n} \left| B_{k,t,z} + R_{k,t,z} \right|^{2^*} dx  = \frac{1}{2^*} \int_{\R^n} B_{k,t,z}^{2^*}dx + \int_{\R^n} B_{k,t,z}^{2^*-1} R_{k,t,z} \,dx \\
+ (2^*-1)  \int_{\R^n} B_{k,t,z}^{2^*-2} R_{k,t,z}^2 \,dx + \smallo\( \mu_k^{\pui}\). 
\end{multline}
As shown in the proof of Lemma \ref{lemmeR}, we have 
 \ben \label{ener4}
  \int_{\R^n} \sum_{i,j=1}^n (h_k)_{ij} \,\partial_i B_{k,t,z} \,\partial_j B_{k,t,z} \,dx  = 0\een
 for any $(t,z) \in [1/A,A] \times B(0,1)$. We now integrate \eqref{eqRk} against $R_{k,t,z}$:  after integrating by parts and using that $h$ is divergence-free, we find that
\begin{multline} \label{ener5}
 \int_{\R^n}| \nabla R_{k,t,z}|^2 dx - (2^*-1) \int_{\R^n} B_{k,t,z}^{2^*-2} R_{k,t,z}^2 \,dx\\
 = \eps_k \int_{\R^n} \sum_{i,j=1}^n (h_k)_{ij} \,\partial_i B_{k,t,z} \,\partial_j R_{k,t,z} \,dx.
\end{multline}
 It remains to plug  \eqref{ener2}, \eqref{ener3}, \eqref{ener3bis}, \eqref{ener4} and \eqref{ener5} into \eqref{ener0}. This gives 
\begin{multline*} 
I\(U_{k,t,z}\circ \Phi \) = \frac{1}{n} K_n^{-n}  - \frac12 \int_{\R^n} |\nabla R_{k,t,z}|^2 dx  \\
 + \frac{\eps_k^2}{4} \int_{\R^n} \sum_{i,j,p=1}^n (h_k)_{ip}(h_k)_{pj}    \,\partial_i B_{k,t,z} \,\partial_j B_{k,t,z} \,dx \\
 - \frac{n-2}{32(n-1)} \eps_k^2 \int_{\R^n} \sum_{i,j,\ell=1}^n (\partial_{i} (h_k)_{j \ell})^2 B_{k,t,z}^2 \,dx + \smallo\(\mu_k^{\pui}\).
\end{multline*}
Simple changes of variables using \eqref{defhk} and \eqref{Rtz} now yield
\begin{multline*} 
\frac{\eps_k^2}{4} \int_{\R^n} \sum_{i,j,p=1}^n (h_k)_{ip}(h_k)_{pj}    \,\partial_i B_{k,t,z} \,\partial_j B_{k,t,z} \,dx\\
 - \frac{n-2}{32(n-1)} \eps_k^2 \int_{\R^n} \sum_{i,j,\ell=1}^n (\partial_{i} (h_k)_{j \ell})^2 B_{k,t,z}^2 \,dx\\
=\eps_k^2 \mu_k^4 \Bigg (\frac14 \int_{\R^n}\hspace{-3pt}\sum_{i,j,p=1}^n\hspace{-2pt}h_{ip}\,h_{pj}    \,\partial_i B_{t,z} \,\partial_j B_{t,z} \,dx- \frac{n-2}{32(n-1)}  \int_{\R^n}\hspace{-3pt}\sum_{i,j,\ell=1}^n \hspace{-2pt}(\partial_{i} h_{j \ell})^2 B_{t,z}^2 \,dx \Bigg) 
\end{multline*}
and 
\[ \frac12 \int_{\R^n} |\nabla R_{k,t,z}|^2 dx = \frac12 \eps_k^2 \mu_k^4 \int_{\R^n} |\nabla R_{t,z}|^2 dx ,\]
where $R_{t,z}$ and $B_{t,z}$ are as in \eqref{Rktz} and \eqref{Btz}. Together with \eqref{defmuk}, \eqref{ener00}, \eqref{intUPhi} and \eqref{ener1}, this concludes the proof of \eqref{DLI}. The expansions of the derivatives with respect to $t$ and $z$ follow from similar estimates.
\end{proof}

To conclude this section, we show that $I\(u_{k,t,z}\)$ expands at first-order as $I\(W_{k,t,z}\)$ and that this expansion is $C^1$ in $t$:

\begin{proposition} \label{DLenergie}
We have
\[I \(u_{k,t,z} \) = I\(W_{k,t,z} \) + \smallo\(\mu_k^{\pui}\)\text{ and }\partial_t \( I \(u_{k,t,z} \) \) = \partial_t \(  I\(W_{k,t,z} \) \)+ \smallo\(\mu_k^{\pui}\) \]
as $k \to  \infty$, uniformly with respect to $(t,z) \in [1/A,A] \times \overline{B(0,1)}$, where $u_{k,t,z}$ is as in Proposition \ref{LS}. 
\end{proposition}

$C^1$-expansions in the $z$ variable also hold true, but we will not need them here. 

\begin{proof}
By the mean value theorem, \eqref{esterreur} and \eqref{errvp}, we have 
\[ \bal I \( u_{k,t,z} \) & = I\(W_{k,t,z} \) + DI \(W_{k,t,z} \)(\vp_k(t,z)) + \bigO \( \Vert \vp_k(t,z) \Vert_{H^1(M)}^2 \) \allowdisplaybreaks\\
& =  I\(W_{k,t,z} \) +  \int_M E_{k,t,z} \vp_k(t,z) \,dv_{\tilde g}  +  \smallo\(\mu_k^{\pui}\)=   I\(W_{k,t,z} \) +  \smallo\(\mu_k^{\pui}\), \eal \]
which proves the first equality. We now show the $C^1$--estimate in $t$. By Proposition \ref{LS}, there exist, for any $(t,z) \in [1/A,A] \times B(0,1)$, real numbers $(\lambda_k^j(t,z))_{0 \le j \le n}$ such that $u_{k,t,z}$ satisfies 
\ben \label{DL00}
 \( \triangle_{\check g} + c_n \Scal_{\check g} \) u_{k,t,z} - f(u_{k,t,z}) =  \sum_{j=0}^n \lambda_k^j(t,z) \( \triangle_{\check g} + c_n \Scal_{\check g} \) \hat{Z}_{j,k,t,z}\text{ in }M.\een
By definition of $u_{k,t,z}$, we have independently
\begin{multline} \label{DL0} 
 \( \triangle_{\check g}  + c_n \Scal_{\check g} \) u_{k,t,z} - f(u_{k,t,z}) = E_{k,t,z} 
 +  \( \triangle_{\check g} + c_n \Scal_{\check g} - f'(W_{k,t,z}) \) \vp_k(t,z) \\
  - \( f(u_{k,t,z}) - f(W_{k,t,z}) - f'(W_{k,t,z}) \vp_k(t,z) \), 
\end{multline}
 where $E_{k,t,z}$ is as in \eqref{erreur}. By \eqref{esterreur} and \eqref{errvp}, Theorem 4.3 of Premoselli \cite{Premoselli13} applies and 
 shows the existence of a sequence $\sigma_k $ of positive numbers with $\sigma_k\to0$ as $k\to\infty$ such that for any $x \in M$ and $(t,z) \in [1/A,A] \times B(0,1)$,
\ben \label{estphik}
 |\vp_k(t,z)(x)| \le \sigma_k \( 1 + U_{k,t,z}\( \Phi(x )\) \) 
 \een
holds. By \eqref{estR} and \eqref{defW}, we can then let $R>0$ be large enough so that 
\[ W_{k,t,z}\( \Phi^{-1}(x) \) \ge \frac12 B_{k,t,z}(x) \]
holds for any $x \in B(\xi_k(z), \sqrt{\mu_k}/R) \subset \R^n$ and for any $k\in\N$ and $(t,z) \in [1/A,A] \times B(0,1)$. In $\Phi^{-1}\(B(\xi_k(z), \sqrt{\mu_k}/R)\)$, we then have
\[ \left| f(u_{k,t,z}) - f(W_{k,t,z}) - f'(W_{k,t,z}) \vp_k(t,z) \right| \le C \(U_{k,t,z}\circ\Phi \)^{2^*-3} \vp_k(t,z)^2. \]
In $M \backslash \Phi^{-1}\(B(\xi_k(z), \sqrt{\mu_k}/R)\)$, \eqref{estphik} similarly shows that 
\[ \left| f(u_{k,t,z}) - f(W_{k,t,z}) - f'(W_{k,t,z}) \vp_k(t,z) \right| \le C |\vp_k(t,z)|. \]
We let $j \in \{0, \dotsc, n\}$ and integrate \eqref{DL0} against $\hat{Z}_{j,k,t,z}$. Since $|\hat{Z}_{j,k,t,z}| \le C \(U_{k,t,z}\circ  \Phi\)$ in $M$, by using the latter inequality together with straightforward computations, we obtain
\be 
\bal 
 &\int_M \( f(u_{k,t,z}) - f(W_{k,t,z}) - f'(W_{k,t,z}) \vp_k(t,z) \) \hat{Z}_{j,k,t,z} \,dv_{\check g} \\
&\qquad =  \bigO \left( \int_{\Phi^{-1} \(B(\xi_k(z), \sqrt{\mu_k}/R)\)} \( U_{k,t,z} \circ\Phi \)^{2^*-2} |\vp_k(t,z)|^2 dv_{\check g} \right) \\
&\qquad\quad+ \bigO \left( \int_{M \backslash \Phi^{-1} \(B(\xi_k(z), \sqrt{\mu_k}/R)\)} \( U_{k,t,z} \circ \Phi \) |\vp_k(t,z)| \,dv_{\check g} \right)\allowdisplaybreaks\\
&\qquad = \bigO \(\Vert \vp_k(t,z) \Vert_{H^1(M)}^2  + \Vert \vp_k(t,z) \Vert_{H^1(M)} \Vert B_{k,t,z} \Vert_{L^{\frac{2n}{n+2}}(\R^n\backslash B(\xi_k(z), \sqrt{\mu_k}/R))} \)\\
&\qquad= \smallo\(\mu_k^{\pui}\)
\eal 
\ee
holds, where we have used \eqref{esterreur} and \eqref{errvp} in the last line. With \eqref{DLhatZ1}, \eqref{DLhatZ2}, \eqref{DL00} and \eqref{DL0}, the latter shows that 
\begin{multline}\label{DL2}
(1+\smallo\(1\)) \Vert \nabla V_j \Vert_{L^2(\R^n)}^2 \lambda_k^j(t,z) =  \int_M E_{k,t,z} \hat{Z}_{j,k,t,z} \,dv_{\check g}  + \smallo\(\mu_k^{\pui}\) \\
+ \int_M \vp_k(t,z) \(\triangle_{\check g} + c_n \Scal_{\check g} - f'(W_{k,t,z}) \) \hat{Z}_{j,k,t,z} \,dv_{\check g}. 
\end{multline}
By \eqref{defhatZ} and Proposition \ref{DLenergie2}, we have 
\ben \label{DL3}
\int_M E_{k,t,z} \hat{Z}_{0,k,t,z} \,dv_{\check g}   =\left\{ \bal&\frac{2}{n-2}\, t \,\partial_t  \( I \(W_{k,t,z} \) \)  = \bigO \(\mu_k^{\pui} \)&&\text{if }j=0 \\
& -n  \, t \,\partial_{z_j} \( I \(W_{k,t,z} \) \)= \bigO \(\mu_k^{\pui} \)&&\text{if }1\le j\le n.\eal\right.
\een
By  \eqref{estR}, \eqref{defW}, \eqref{defhatZ} and \eqref{althZ}, we have, for $0 \le j \le n$ and $x \in B(y_k,2r_k)$,
\begin{multline*} 
\(\triangle_{\check g} + c_n \Scal_{\check g} - f'(W_{k,t,z}) \) \hat{Z}_{j,k,t,z}(\Phi^{-1}(x))  \\
 = \(\triangle_{\tilde g} + c_n \Scal_{\tilde g} - f'(B_{k,t,z}) \) \( Z_{j,k,t,z} + Q_{j,k,t,z} \)(x) + \bigO \(r_k^{-n}\mu_k^{\pui} \) \\
 + \bigO\(\frac{\eps_k \mu_k^{\frac{n+2}{2}}}{(\mu_k + |\xi_k(z) - x|)^{n}}\) +  \bigO \(\left \{ \bal &B_{k,t,z}(x)^{2^*-2} &&  \textrm{if } |x - \xi_k(z)| \le \sqrt{\mu_k(t)}  \\ & B_{k,t,z}(x) && \textrm{otherwise }\eal \right\} \). 
\end{multline*}  
Using then \eqref{calc1}, \eqref{calc2}, \eqref{calc3}, \eqref{lapQ} and \eqref{estQ}, we find that, for $x \in B(y_k, 2r_k)$,
$$\(\triangle_{\tilde g} + c_n \Scal_{\tilde g} - f'(B_{k,t,z}) \) \( Z_{j,k,t,z} + Q_{j,k,t,z} \)(x) 
= \bigO \(\frac{\eps_k \mu_k^{\pui}}{(\mu_k + |\xi_k(z) - x|)^{n-2}}\)$$
holds. Thus, 
\begin{multline}
 \left| \(\triangle_{\check g} + c_n \Scal_{\check g} - f'(W_{k,t,z}) \) \hat{Z}_{j,k,t,z}\( \Phi^{-1}(x) \) \right| \\ 
\le C\(r_k^{-n}\mu_k^{\pui}+ \eps_k B_{k,t,z}(x) +  \left \{ \bal & B_{k,t,z}(x)^{2^*-2}  &&\textrm{if } |x - \xi_k(z)| \le \sqrt{\mu_k(t)} \\ & B_{k,t,z}(x)  &&\textrm{otherwise }\eal \right\}\)
\end{multline}
holds for any $x \in B(y_k, 2r_k)$. As a consequence, with \eqref{defmuk},
\[ \bal \left \Vert \Big(\triangle_{\check g} + c_n \Scal_{\check g} - f'(W_{k,t,z}) \Big) \hat{Z}_{j,k,t,z} \right \Vert_{L^{\frac{2n}{n+2}}(M)} & \le C\(r_k^{-n}\mu_k^{\pui} + \eps_k \mu_k^2 + \mu_k^{\frac{n+2}{4}}\)\\
& \le C \mu_k^{\frac{n-2}{4}} .
\eal \]
Hence, with \eqref{esterreur} and \eqref{errvp}, H\"older's inequality shows that
\ben \label{DL31}
 \left |  \int_M \vp_k(t,z) \(\triangle_{\check g} + c_n \Scal_{\check g} - f'(W_{k,t,z}) \) \hat{Z}_{j,k,t,z} \,dv_{\check g}\right| = \smallo \(\mu_k^{\pui}\) .
 \een
Together with \eqref{DL2} and \eqref{DL3}, this shows that 
\ben \label{DL4}
\sum_{j=0}^n |\lambda_k^j(t,z)| = \bigO \( \mu_k^{\pui} \). 
\een
We can now use \eqref{DL00} to write that 
\[ \bal \partial_t \(I\(u_{k,t,z} \)\) & = DI \(u_{k,t,z} \) \(\partial_t u_{k,t,z} \) \\
& =  \sum_{j=0}^n \lambda_k^j(t,z) \big( \hat{Z}_{j,k,t,z}, \partial_t W_{k,t,z} + \partial_t \vp_k(t,z) \big)_{H^1(M)}, 
\eal \]
where $W_{k,t,z}$ is as in \eqref{defW}. On the one side, with  \eqref{Z0}, \eqref{Zi}, \eqref{DLhatZ1} and \eqref{DLhatZ2}, we have 
\[  \big( \hat{Z}_{j,k,t,z}, \partial_t W_{k,t,z} \big)_{H^1(M)} = \frac{n-2}{2t}  \delta_{j0} \Vert \nabla V_0\Vert_{L^2(\R^n)}^2 + \smallo\(1\) \]
as $k \to\infty$, where $\delta_{j0}$ is the Kronecker symbol. On the other side, since $\vp_k(t,z) \in \hat{K}_{k,t,z}^{\perp}$, we have $\big(\vp_k(t,z), \hat{Z}_{j,k,t,z} \big)_{H^1(M)} = 0$ for all $0 \le j \le n$, hence
\[ \big( \hat{Z}_{j,k,t,z}, \partial_t \vp_k(t,z) \big)_{H^1(M)} = - \big( \partial_t \hat{Z}_{j,k,t,z}, \vp_k(t,z) \big)_{H^1(M)} = \smallo\(1\), \]
where we have used  \eqref{esterreur}, \eqref{errvp}, \eqref{DLhatZ1} and \eqref{DLhatZ2}. Together with \eqref{DL4}, we thus obtain that
\be  \bal \partial_t \(I\(u_{k,t,z} \)\) & =   \frac{n-2}{2t}  \delta_{j0} \Vert \nabla V_0\Vert_{L^2(\R^n)}^2 \lambda_k^0(t,z) + \smallo\(\mu_k^{\pui}\). 
\eal \ee
Together with \eqref{DL2}, \eqref{DL3} and \eqref{DL31}, this becomes
\be \bal \partial_t \(I\(u_{k,t,z} \)\) & =  \partial_t \(I\( W_{k,t,z} \)\) + \smallo\(\mu_k^{\pui}\), 
\eal \ee
which concludes the proof of Proposition \ref{DLenergie}. 
\end{proof}

\section{Proof of Theorem \ref{Th1} -- Part $2$: Finding critical points of the reduced energy} \label{secreduced}

Let $n \ge 11$ and let $W: (\R^n)^4 \to \R$ be a four-linear form in $\R^n$ possessing the same symmetries as a Weyl tensor, that is $W_{ijk\ell} = - W_{jik\ell} = - W_{ij \ell k} = W_{k\ell ij}$,  $\sum_{i=1}^n W_{iji\ell} = 0$ and 
\[ W_{ijk\ell} + W_{jki\ell} + W_{kij\ell} = 0 \]
for any $i,j,k,\ell \in \{1, \dotsc, n\}$. Assume that $W$ is not identically zero, that is
\[ |W|^2  = \sum_{i,j,k,\ell=1}^n (W_{ijk\ell})^2 >0. \]
For any $k,\ell \in\{1, \dotsc, n\}$, let 
\ben \label{Tkl}
 T_{k\ell} := \sum_{p,q,r=1}^n \( W_{kpqr} + W_{krqp} \)\( W_{\ell pqr} + W_{\ell rqp} \).  
 \een
Straightforward computations using the symmetries of $W$ show that 
\[ T_{k\ell} = 3 \sum_{p,q,r=1}^n W_{kpqr} W_{\ell pqr} \quad \textrm{ and } \quad    \sum_{k=1}^n T_{kk} = 3|W|^2 > 0. \]
We let, for $x \in \R^n$ and $1 \le i,j \le n$, 
\ben \label{defh}
h(x)_{ij} := \frac13 \sum_{p,q=1}^n  W_{ipjq}x_p x_q.
\een
Using the symmetries of $W$, it is easily seen that $h$ satisfies \eqref{hyph1}. We define, for $t >0$ and $z \in B(0,1)$,
\begin{align}\label{defF}
F(t,z) & := \frac14 \int_{\R^n} \sum_{i,j,p=1}^n h_{ip}\,h_{pj} \,\partial_i B_{t,z} \,\partial_j B_{t,z} \,dx - \frac{n-2}{32(n-1)} \int_{\R^n}\hspace{-1pt}\sum_{i,j,\ell=1}^n (\partial_i h_{j\ell})^2 B_{t,z}^2 \,dx \nonumber\\
&\quad - \frac12 \int_{\R^n}|\nabla R_{t,z}|^2 dx + \Lambda(n) \check{u_0}(x_0) t^{\pui} \nonumber\allowdisplaybreaks\\
& =: F_1(t,z) + F_2(t,z) + F_3(t,z) + \Lambda(n) \check{u}_0(x_0)t^{\pui},
\end{align}
where $\check{u}_0$ is as in \eqref{deftu0} and $x_0$ is as in the beginning of Section \ref{secLS}, $R_{t,z}$ is defined by \eqref{Rktz} for $h$ given by \eqref{defh}, $B_{t,z}$ is as in \eqref{Btz} and $\Lambda(n)$ is the positive numerical constant given by \eqref{defLa}. The function $F$ is smooth in $ (0,  \infty) \times B(0,1)$. By \eqref{Rtz}, the uniqueness of $R_{t,z}$ and since $h(-x) = h(x)$ for all $x \in \R^n$, we have, for any $t >0, z \in B(0,1)$ and $x \in \R^n$, $R_{t,-z}(x) = R_{t,z}(-x)$ and thus $F(t,-z) = F(t,z)$. This shows that, for all $t >0$ and $1 \le i,j,k \le n$, 
\ben \label{dercroisees}
\partial_t\partial_{z_i} F(t,z)_{|z=0} = 0 \quad \textrm{ and } \quad  \partial_t\partial^3_{z_i z_j z_k} F(t,z)_{|z=0} = 0.
\een
For real numbers $p,q \ge 0$ with $p > q+2$, we define
\[ I_p^q := \int_{0}^{ \infty} \frac{r^q}{\(1+r\)^p } dr. \]
The following induction formulas are easily proven (see Aubin \cite{Aubin2}):
\ben \label{indu} I_{p+1}^q = \frac{p-q-1}{p}I_p^q \quad \textrm{ and } \quad I_{p+1}^{q+1} =  \frac{q+1}{p-q-1}I_{p+1}^q. \een
From \eqref{indu} the following expressions are easily obtained:
\begin{multline}\label{indu2}
I_n^{\frac{n}{2}}  = \frac{n}{n-2} I_n^{\pui}, \quad I_n^{\frac{n+2}{2}} = \frac{n+2}{n-4}\frac{n}{n-2} I_n^{\pui},\, 
I_{n+2}^{\frac{n+4}{2}}  = \frac{(n+4)(n+2)}{4(n-2)(n+1)} I_n^{\pui},\\ I_{n+1}^{\frac{n+2}{2}}  =\frac{n+2}{2(n-2)}I_n^{\pui}, \,
I_{n-2}^{\pui}  = \frac{4(n-1)}{n-4}I_n^{\pui} , \, I_{n-2}^{\frac{n}{2}} = \frac{4n(n-1)}{(n-4)(n-6)} I_n^{\pui}.
\end{multline}
Recall that, for any $t>0$ and $z \in \R^n$, we have $\int_{\R^n} B_{t,z}^{2^*} \,dx = K_n^{-n}$. A simple change of variables then yields 
\begin{equation}\label{In}
 I_{n}^{\pui} = \frac{2}{\omega_{n-1}}\( n(n-2) \)^{- \frac{n}{2}} K_n^{-n},
\end{equation}
 where $\omega_{n-1}$ is the volume of the standard unit sphere of dimension $n-1$. This allows us to compute the integrals in \eqref{indu2}. 

\smallskip
We first prove the following lemma which shows the existence of a strict global minimum of $t \mapsto F(t,0)$:

\begin{lemma} \label{lemmez0}
Assume that $n \ge 11$. Then $t\in[0, \infty) \mapsto F(t,0)$ possesses a unique global minimum $t_0 >0$ such that $F(t_0,0) < 0$. This minimum is also non-degenerate, meaning that $\partial_t^2 F(t_0, 0) >0$.
\end{lemma}

\begin{proof}
First, by \eqref{Rtz} and the uniqueness of $R_{t,z}$, we have $R_{t,0} = 0$, hence $F_3(t,0) =0$ for all $t>0$. Let $t >0$ and $z \in B(0,1)$. Since $\sum_{j=1}^n x_j h(x)_{ij} = 0$ for all $1 \le i \le n$ and by definition of $B_{t,z}$, we have
\begin{align}\label{F1}
&\frac14 \sum_{i,j,p =1}^n h(x)_{ip}h(x)_{pj} \,\partial_i B_{t,z}(x) \,\partial_j B_{t,z} (x)\nonumber\\
&\qquad= \frac{1}{4 n^2} \frac{t^{n-2}}{\(t^2 + \frac{|x-z|^2}{n(n-2)}\)^{n}} \sum_{i,j,p=1}^n h(x)_{ip} h(x)_{pj} z_i z_j.  
\end{align}
Thus $F_1(t,0) = 0$ and $F(t,0) = F_2(t,0) + \Lambda(n)\check{u}_0(x_0) t^{\pui}$ for any $t >0$. By \eqref{defh}, we have, for all $x \in \R^n$, 
\ben\label{partialih}
\partial_i h_{j\ell}(x) = \frac13 \sum_{p=1}^n (W_{j i \ell p} + W_{j p \ell i })x_p.
\een 
Hence, by \eqref{defF}, we have 
\[ F_2(t,0) = -  \frac{n-2}{32(n-1)} t^4  \int_{\mathbb{S}^{n-1}}  \sum_{i,j,\ell=1}^n (\partial_i h_{j\ell}(x))^2 d \sigma(x) \cdot \int_0^{\infty} \frac{r^{n+1}dr}{\big(1+\frac{r^2}{n(n-2)}\big)^{n-2}}. \]
 Recall that 
\ben \label{intS1}
\int_{\mathbb{S}^{n-1}} x_i x_j\,d \sigma(x) = \frac{\omega_{n-1}}{n} \delta_{ij} 
\een
for all $i,j \in \{1, \dotsc, n\}$. As a consequence, we have 
$$ \int_{\mathbb{S}^{n-1}}  \sum_{i,j,\ell=1}^n (\partial_i h_{j\ell}(x))^2 d \sigma(x) = \frac{\omega_{n-1}}{9n} \sum_{k=1}^n T_{kk}, $$
where $T_{kk}$ is as in \eqref{Tkl}. Independently, straightforward computations with \eqref{indu2} and \eqref{In} give that
\[ \int_0^{\infty} \frac{r^{n+1} dr}{\big(1+\frac{r^2}{n(n-2)}\big)^{n-2}} = \frac{4n^2(n-1)(n-2)}{(n-4)(n-6)}\frac{K_n^{-n}}{\omega_{n-1}}. \]
This shows that 
\ben \label{DLWeyl}
 F(t,0) = - \(  \frac{n(n-2)^2}{72(n-4)(n-6)}K_n^{-n}\sum_{k=1}^n T_{kk} \) t^4 + \Lambda(n) \check{u}_0(x_0) t^{\pui}. 
 \een
Since $\Lambda(n) >0$, $\check{u}_0>0$ in $M$, $\sum_{k=1}^n T_{kk} >0$ and  $n \ge 11$, Lemma \ref{lemmez0} follows.
\end{proof}

The value of $t_0$ only depends on $n,W$ and $\check{u}_0$, that have been fixed from the beginning. From now on, we will let $A>0$ be chosen large enough so that $t_0 \in [2/A, A/2]$. We now show that the second derivatives of $F$ at $(t,0)$ in $z$ vanish for every $t >0$:

\begin{proposition} \label{propder2}
Let $n \ge 11$ and let $t >0$ be fixed. Then, for any $1 \le k, \ell \le n$,
\[ \partial^2_{z_k z_\ell} F(t,z)_{|z=0}  =0.\]
\end{proposition}

\begin{proof}
Let $t>0$ and $k,\ell \in \{1, \dotsc, n \}$ be fixed throughout this proof. 
First, by \eqref{Rtz} and the uniqueness of $R_{t,z}$, we have $R_{t,0} = 0$ and $\partial_{z_i}\( R_{t,z}\)_{|z=0}  = 0$ for all $1 \le i \le n$. As a consequence,
$ \partial^2_{z_k z_\ell} F_3(t,z)_{|z=0}=0$. On the one hand, with \eqref{F1} and differentiating under the integral, we get that 
\begin{align*}
\partial^2_{z_k z_\ell}  F_1(t,z)_{|z=0}&= \frac{1}{2n^2} \int_{\R^n }  \sum_{p=1}^n h(x)_{kp}\,h(x)_{p\ell} \frac{t^{n-2}dx}{\big(t^2 + \frac{|x|^2}{n(n-2)}\big)^{n}} .\\
&  = \frac{t^2}{2 n^2} \int_{\mathbb{S}^{n-1}}  \sum_{p=1}^n h(x)_{kp}\,h(x)_{p\ell}\,d \sigma(x) \cdot \int_{0}^{ \infty} \frac{r^{n+3} dr}{\big(1+ \frac{r^2}{n(n-2)}\big)^n}. 
\end{align*}
By using the symmetries of $W$, it is easily seen that
\[ \sum_{p,i,j=1}^n W_{kpij} \( W_{\ell pij} + W_{\ell jip} \) = \frac12 T_{k \ell}, \]
where $T_{k \ell}$ is as in \eqref{Tkl}. Using \eqref{defh} and Corollary 29 of Brendle \cite{Brendle}, we thus obtain 
$$\int_{\mathbb{S}^{n-1}}  \sum_{p=1}^n h(x)_{kp}\,h(x)_{p\ell}\,d \sigma(x) = \frac{\omega_{n-1}}{18n(n+2)} T_{k\ell}  .$$
Together with \eqref{indu2} and \eqref{In}, we then obtain 
\ben \label{F11}
\partial^2_{z_k z_\ell} F_1(t,z)_{|z=0} =  \frac{t^2}{36} \frac{n-2}{n-4} K_n^{-n}T_{k\ell}. 
\een
By using \eqref{partialih} together with change of variables, we obtain
\[  F_2(t,z) = -  \frac{n-2}{288(n-1)} \int_{\R^n} \sum_{p,q=1}^n T_{pq}(x+z)_p (x+z)_q \frac{t^{n-2} dx}{\big(t^2+\frac{|x|^2}{n(n-2)}\big)^{n-2}}, \]
so that differentiating and using \eqref{indu2}, \eqref{In} and \eqref{intS1} yields
\ben \label{F22} \bal  \partial^2_{z_k z_\ell} F_2(t,z)_{|z=0} & =  - \frac{n-2}{144(n-1)} T_{k \ell}  \int_{\R^n} \frac{t^{2}dx}{\big(1+\frac{|x|^2}{n(n-2)}\big)^{n-2}} = -  \frac{t^2}{36} \frac{n-2}{n-4} K_n^{-n}T_{k\ell}. \eal \een
Together with \eqref{F11}, this concludes the proof of Proposition \ref{propder2}. 
\end{proof}

In order to prove that $(t_0,0)$ is a non-degenerate critical point of $F$, we now expand, for a fixed value of $t$, $z \mapsto F(t,z)$ to the fourth order as $z \to 0$. A first, obvious, remark is that for $i,j,k,l \in \{1, \dotsc, n\}$, we have 
\ben \label{der4F2}
\partial_{z_i z_j z_k z_l}^4 F_2(t,z)_{|z=0} = 0.
\een
This easily follows from the expression of $F_2$ given by \eqref{defF}. Indeed, by \eqref{defh},  $\sum_{i,j,\ell=1}^n (\partial_i h_{j\ell}(x))^2$ is a second-order polynomial in $x$, hence $F_2(t,z)$ is a second-order polynomial in $z$ by a change of variables. We now expand $F_1+F_2$: 

\begin{proposition} \label{propF1}
Let $t >0$. We have, as $z \to 0$,
\[ F_1(t,z) + F_2(t,z) = F_2(t,0)  +\frac{K_n^{-n}}{4 n} \sum_{p,q=1}^n \big(h_{pq}(z) \big)^2 
+ \bigO\(|z|^5\).\]
\end{proposition}

\begin{proof}
By definition, $F_1$ and $F_2$ are even in $z$. The expansions to second-order thus follow from $F_1(t,0) = 0$, \eqref{F11} and \eqref{F22}. By \eqref{der4F2}, we only have to expand $F_1(t,z)$ to fourth order. By \eqref{defF}, \eqref{F1} and a change of variables, we have 
\[ F_1(t,z) =  \frac{1}{4n^2} \int_{\R^n} \frac{t^{n-2}}{\big(t^2 + \frac{|x|^2}{n(n-2)}\big)^{n}}  \sum_{i,j,p=1}^n h(x+z)_{ip}\,h(x+z)_{pj} z_i z_j\,dx.\]
It is easily seen that, for any $x,z \in \R^n$ and $1\le a,b,c,d\le n,$
\begin{multline*}
 (x+z)_a(x+z)_b(x+z)_c(x+z)_d  =x_ax_bx_cx_d + x_ax_bx_c z_d + x_ax_bx_d z_c + x_ax_cx_d z_b  \\
+ x_bx_cx_d z_a+ x_ax_b z_c z_d + x_a x_c z_b z_d + x_a x_d z_b z_c + x_b x_c z_a z_d + x_b x_d z_a z_c + x_c x_d z_a z_b\\+ \bigO\(|z|^3|x|+|z|^4\).
\end{multline*}
As a consequence and by \eqref{defh}, we have, as $z \to 0$,
$$F_1(t,z) + F_2(t,z)=  F_2(t,0)+G(z) + \bigO\(|z|^5\),$$
where
\begin{multline*}
G(z):=\frac{1}{36n^2} \int_{\R^n} \(1 + \frac{|x|^2}{n(n-2)}\)^{-n} \sum_{i,j,p,a,b,c,d=1}^n W_{iapb}W_{pcjd} z_iz_j( x_ax_b z_c z_d  \\
+ x_a x_c z_b z_d+ x_a x_d z_b z_c + x_b x_c z_a z_d + x_b x_d z_a z_c + x_c x_d z_a z_b)\,dx.   \end{multline*}
By \eqref{indu2}, \eqref{In} and \eqref{intS1}, we have, for $1 \le a,b \le n$,
\[  \int_{\R^n} \frac{x_ax_b \,dx}{\big(1+ \frac{|x|^2}{n(n-2)}\big)^{n}}  = n K_n^{-n} \delta_{ab}. \]
Since $W$ is totally traceless, we thus obtain 
\begin{multline}\label{der41}
G(z) =  \frac{1}{36n}K_n^{-n} \sum_{i,j,p,a,b,c=1}^n \big( W_{iapb} W_{pajc} z_bz_cz_iz_j + W_{iapb} W_{pcja} z_b z_c z_i z_j \\
 + W_{iapb}W_{pbjc} z_a z_c z_i z_j +  W_{iapb} W_{pcjb} z_a z_c z_i z_j\big).
\end{multline} 
Since $W_{pajc}=-W_{pacj}$, we have $\sum_{i,j,p,a,b,c=1}^n W_{iapb} W_{pajc} z_bz_cz_iz_j=0$. Similarly, we obtain that the last two terms in the right-hand side of \eqref{der41} vanish. Since moreover $W_{pcja} = W_{japc} $, we then obtain
\begin{align*}
 G(z)& = \frac{1}{36n} K_n^{-n} \sum_{i,j,p,a,b,c=1}^n W_{iapb} W_{japc} z_b z_c z_i z_j \\
& =  \frac{1}{36n} K_n^{-n} \sum_{a,p=1}^n \( -\sum_{i,b=1}^n W_{iabp} z_i z_b \)^2 =\frac{1}{4n} K_n^{-n} \sum_{a,p=1}^n \( h_{ap}(z) \)^2 
 \end{align*} 
 by \eqref{defh}, which concludes the proof of Proposition \ref{propF1}. 
\end{proof}
 
We next expand $F_3(t,z)$ at fourth-order in $z$ at $(t,0)$ for some fixed $t >0$:

\begin{proposition} \label{propF3}
Let $t >0$. As $z \to 0$, we have 
\[ F_3(t,z) = - \frac{n+4}{48(n+1)}K_n^{-n} \sum_{p,q=1}^n \(h_{pq}(z)\)^2 + \bigO\(|z|^5\).\]
\end{proposition}

\begin{proof}
By \eqref{Rtz}, for $(t,z) \in (0, \infty) \times \R^n$, we have $R_{t,z} = \tilde{R}_{t,z}(x-z)$, where $\tilde{R}_{t,z}$ is the unique solution in $K_{t,0}^\perp$ of 
\ben \label{derF31}
 \triangle_{\delta_0} \tilde{R}_{t,z}(x) - (2^*-1) B_{t,0}(x)^{2^*-2} \tilde{R}_{t,z}(x) = - \frac{t^{\pui}}{n(n-2)}  \frac{ \sum_{p,q=1}^n h_{pq}(x+z) z_p z_q}{\big(t^2 + \frac{|x|^2}{n(n-2)} \big)^{\frac{n+2}{2}}}, \een
 where $h$ is given by \eqref{defh}. Hence, for fixed $t >0$, all the derivatives of $R_{t,z}$ with respect to $z$, taken at $z=0$, belong to $K_{t,0}^{\perp}$. We let in what follows, for $a,b \in \{1, \dotsc, n \}$, 
\[ L_{ab}: = \partial^2_{z_a z_b}\( R_{t,z}\)_{|z=0} \in K_{t,0}^{\perp}.\]
Since $R_{t,0} = 0$ and $\partial_{z_i} \(R_{t,z}\)_{|z=0}  = 0$ for all $1 \le i \le n$, it is easily seen that, for $1 \le a,b,c,d \le n$, we have
\ben \label{derF32}
\begin{aligned}
\partial_{z_a} F_3(t,z)_{|z=0} & = 0, \quad \partial^{2}_{z_a z_b} F_3(t,z)_{|z=0}  = 0, \quad \partial^{3}_{z_a z_b z_c} F_3(t,z)_{|z=0}  = 0 \quad \text{ and } \\
\partial_{z_a z_b z_c z_d}^4 F_3(t,z)_{|z=0} & = - \int_{\R^n} \( \langle \nabla L_{ab}, \nabla L_{cd} \rangle +  \langle \nabla L_{ac}, \nabla L_{bd} \rangle +  \langle \nabla L_{ad}, \nabla L_{bd} \rangle \) dx.
\end{aligned}
\een
Differentiating \eqref{derF31} shows that $L_{ab}$ satisfies
\[   \triangle_{\delta_0} L_{ab}(x) - (2^*-1) B_{t,0}^{2^*-2} L_{ab}(x) = -\frac{2t^{\pui}}{n(n-2)}  \frac{h_{ab}(x)}{\big(t^2 + \frac{|x|^2}{n(n-2)} \big)^{\frac{n+2}{2}}} . \]
In particular, for any $x \in \R^n$, we have $L_{ab}(x) = t^{1 - \frac{n}{2}} L_{ab}^0(x/t)$, where $L_{ab}^0 \in K_{1,0}^\perp$ satisfies
\ben \label{derF33}
\triangle_{\delta_0} L_{ab}^0(x) - (2^*-1) B_{1,0}^{2^*-2} L_{ab}^0(x)  = -\frac{2}{n(n-2)}  \frac{h_{ab}(x)}{\big(1 + \frac{|x|^2}{n(n-2)} \big)^{\frac{n+2}{2}}}.
\een
Coming back to \eqref{derF32}, we also have, with a change of variables,
\ben \label{derF34}
\partial_{z_a z_b z_c z_d}^4 F_3(t,z)_{|z=0} = - \int_{\R^n} \( \langle \nabla L_{ab}^0, \nabla L_{cd}^0 \rangle +  \langle \nabla L_{ac}^0, \nabla L_{bd}^0 \rangle +  \langle \nabla L_{ad}^0, \nabla L_{bd}^0 \rangle \) dx.
\een
We first find an explicit expression for $L_{ab}^0$. Since $W$ is totally traceless, $h$ is a harmonic polynomial in $\R^n$. As a consequence, we can look for a solution of \eqref{derF33} under the form 
\[ L_{ab}^0(x) = h_{ab}(x) f(|x|^2) 
. \]
Simple computations show that $f(t) = - \frac{1}{n} \big(1+ \frac{t}{n(n-2)}\big)^{-n/2}$ satisfies \eqref{derF33} for all $t \in \R$. By using \eqref{defh}, we thus obtain
\ben \label{derF35}
L_{ab}^0(x) = - \frac{1}{n} \frac{h_{ab}(x)}{\big(1+ \frac{|x|^2}{n(n-2)}\big)^{\frac{n}{2}}}= - \frac{1}{3n} \frac{\sum_{p,q=1}^n W_{apbq}x_px_q}{\big(1+ \frac{|x|^2}{n(n-2)}\big)^{\frac{n}{2}}}
\een
for $1 \le a,b \le n$ and all $x \in \R^n$. Note that $L_{ab}^0$ given by \eqref{derF35} belongs to $K_{1,0}^\perp$, which follows again from the symmetries of $W$, and is thus the unique solution of \eqref{derF33} in $K_{1,0}^\perp$. Straightforward computations with \eqref{derF35} then show that 
\ben \label{derF352}
\int_{\R^n}  \langle \nabla L_{ab}^0, \nabla L_{cd}^0 \rangle \,dx = K_1 + K_2 + K_3, \een
where
\[ \begin{aligned}
A_1 & := \frac{1}{9n^2} \int_{\R^n} \(1+ \frac{|x|^2}{n(n-2)}\)^{-n} \sum_{i,p,q=1}^n (W_{aibp} + W_{apbi})(W_{cidq} + W_{cqdi})x_p x_q \,dx,\allowdisplaybreaks\\
A_2 & :=  \frac{1}{n^2(n-2)^2}  \int_{\R^n} \(1+ \frac{|x|^2}{n(n-2)}\)^{-n-2}|x|^2h_{ab}(x)\,h_{cd}(x) \,dx,\allowdisplaybreaks\\
A_3 & := - \frac{4}{n^2(n-2)} \int_{\R^n} \(1+ \frac{|x|^2}{n(n-2)}\)^{-n-1} h_{ab}(x) h_{cd}(x) \,dx. 
\end{aligned} \]
By \eqref{indu2}, \eqref{In} and \eqref{intS1}, we have 
\[ \bal A_1 & = \frac{1}{9n} K_n^{-n} \sum_{p,q=1}^n \( W_{apbq} + W_{aqbp}\) \(W_{cpdq} + W_{cqdp} \) \\
& = \frac{2}{9n} K_n^{-n}  \sum_{p,q=1}^n W_{apbq}\(W_{cpdq} + W_{cqdp} \). \eal \]
Using again Corollary 29 of Brendle \cite{Brendle}, we have 
\[ \int_{\mathbb{S}^{n-1}} h_{ab}(x)\,h_{cd}(x)\,d \sigma(x) = \frac{1}{9n(n+2)} \omega_{n-1} \sum_{p,q=1}^nW_{apbq} \( W_{cpdq} + W_{cqdp} \).\]
Together with \eqref{indu2} and \eqref{In}, we thus obtain
\[ \begin{aligned}
A_2 & = \frac{1}{n^2(n-2)^2} \int_0^{\infty} \frac{r^{n+5}}{(1+\frac{r^2}{n(n-2)})^{n+2}}dr \cdot\int_{\mathbb{S}^{n-1}} h_{ab}(x)\,h_{cd}(x)\,d \sigma(x) \\
& = \frac{n+4}{36(n+1)}K_n^{-n} \sum_{p,q=1}^n W_{apbq}\(W_{cpdq} + W_{cqdp} \).
\end{aligned} \]
Similarly, we get
\[ A_3  = - \frac{2}{9n} K_n^{-n} \sum_{p,q=1}^n W_{apbq}\(W_{cpdq} + W_{cqdp} \). \]
Combining the latter equalities with \eqref{derF352} finally shows that 
\be  \int_{\R^n}  \langle \nabla L_{ab}^0, \nabla L_{cd}^0 \rangle \,dx  = \frac{n+4}{36(n+1)} K_n^{-n}  
 \sum_{p,q=1}^n W_{apbq}\(W_{cpdq} + W_{cqdp} \). \ee
Going back to \eqref{derF34}, we have thus proven that  
\begin{multline} \label{derF36}
 \partial_{z_a z_b z_c z_d}^4  F_3(t,z)_{|z=0}  = - \frac{n+4}{36(n+1)} K_n^{-n} \sum_{p,q=1}^n  (W_{apbq}(W_{cpdq} + W_{cqdp} ) \\
+ W_{apcq}(W_{bpdq} + W_{bqdp} )  + W_{apdq}(W_{bpcq} + W_{bqcp} )).
\end{multline}
Using the symmetries of $W$, it is now easily seen that 
\[ \sum_{a,b,c,d,p,q=1}^n W_{apbq}\( W_{cpdq} + W_{cqdp}\) z_az_bz_cz_d = 18 \sum_{p,q=1}^n \( h_{pq}(z)\)^2. \]
A Taylor formula for $F_3$ at fourth order in $(t,0)$ thus shows with \eqref{derF36} that 
\[ \begin{aligned}
F_3(t,z) & = \frac{1}{24} \sum_{a,b,c,d=1}^n \partial_{z_a z_b z_c z_d}^4  F_3(t,z)_{|z=0} z_a z_b z_c z_d + \bigO\(|z|^5\) \\
& =-  \frac{n+4}{48(n+1)} K_n^{-n} \sum_{p,q=1}^n \(h_{pq}(z) \)^2  + \bigO\(|z|^5\), 
\end{aligned} \]
which concludes the proof of Proposition \ref{propF3}.
\end{proof}

Combining Propositions \ref{propder2}, \ref{propF1} and \ref{propF3} and since $F$ is even in $z$, we have therefore proven that, for any $t >0$ fixed,  
\ben \label{Forder4}
 F(t,z) = F(t,0) -  \frac{n^2-8n-12}{48n(n+1)} K_n^{-n} \sum_{p,q=1}^n \( h_{pq}(z) \)^2  + \bigO\(|z|^5\) 
 \een
as $z \to 0$. It is easily checked that $n^2 - 8n - 12 >0$ for $n \ge 11$. We are now in position to conclude the proof of Theorem \ref{Th1}. 

\begin{proof}[Proof of Theorem \ref{Th1}.]
Let $(M,\hat{g})$ be a locally conformally flat manifold that is $Y$-non-degenerate in the sense of Definition \ref{nondege}. Let $W: (\R^n)^4 \to \R$ be a four-linear form as in Section \ref{secreduced} and $h$ be defined by \eqref{defh}. By the symmetries of $W$, $h$ satisfies \eqref{hyph1}. We assume that $W$ is chosen so that 
\ben \label{condih}
 \sum_{p,q=1}^n \( h_{pq}(z) \)^2  \ge C |z|^4 \quad \text{ for all } z \in \R^n   
\een
for some constant $C >0$ independent of $z$. Assumption \eqref{condih} is for example satisfied when $W$ has the following diagonal form:
\[ W_{ijk\ell} = \frac{A_{ij}}{2} \(\delta_{ik} \delta_{j \ell} - \delta_{jk} \delta_{i \ell} \), \]
where $A$ is a nonzero symmetric matrix satisfying $A_{ii} = 0$ and $\sum_{j=1}^n A_{ij} = 0$ for any $1 \le i \le n$ and $A_{ij} \neq 0$ whenever $i \neq j$. Let $(\eps_k)_k, (\mu_k)_k$ and $(r_k)_k$ be sequences as in Section \ref{secLS}, $\tilde g$ and $g$ be defined as in \eqref{deftg} and $\check g$ and $\check{u}_0$ be defined as in \eqref{deftu0}. The analysis of Section \ref{secLS} applies. For $(t,z) \in [1/A,A] \times B(0,1)$,  we let $u_{k,t,z} := W_{k,t,z} + \vp_k(t,z)$, where $W_{k,t,z}$ is as in \eqref{defW} and $\vp_k(t,z)$ is given by Proposition \ref{LS}. For $k\in\N$, we let $F_k: [1/A,A] \times B(0,1) \to \R$ be defined by 
\[ F_k(t,z) := \mu_k^{1-\frac{n}{2}} \Big( I\(u_{k,t,z} \) - \frac1n K_n^{-n} - I(\check{u}_0) \Big).\]
By Propositions \ref{DLenergie2} and \ref{DLenergie}, we have 
\ben \label{convFk}
 F_k(t,z) = F(t,z) + \Lambda_k^0(t,z) \quad \text{ and } \quad \partial_t F_k(t,z) = \partial_t F(t,z) + \Lambda_k^1(t,z) 
 \een
uniformly with respect to $(t,z) \in  [1/A,A] \times B(0,1)$, where $\Lambda_k^0$ and $\Lambda_k^1$ converge uniformly to $0$ in compact subsets of $[1/A,A] \times B(0,1)$ and $F$ is as in \eqref{defF}. 
By Lemma \ref{lemmez0}, \eqref{Forder4} and \eqref{condih}, $F$ has a saddle-type geometry on $[1/A,A] \times B(0,1)$, we can let $\eps$ and $\eta$ be small enough  so that
\[ \min_{t \in [t_0-\eta, t_0+\eta]} F(t,0) = F(t_0,0) >  \max_{t \in [t_0-\eta, t_0 + \eta], |z| = \eps}  F(t, z)\]
and  $\partial_t F(t_0- \eta,z) <0$ and $\partial_t F(t_0+ \eta,z) >0$ for any $z\in B(0,\eps)$. Since $F(t_0,0) <0$, by decreasing $\eps$ and $\eta$ if necessary, we can also assume that
\[ \max_{(t,z)\in [t_0-\eta, t_0+\eta] \times \overline{B(0,\eps)}} F(t,z) <0. \]
By \eqref{convFk}, choosing $k$ large enough, we again have 
\[ \min_{t \in [t_0-\eta, t_0+\eta]} F_k(t,0) >  \max_{t \in [t_0-\eta, t_0 + \eta], |z| = \eps}  F_k(t, z), \]
and  $\partial_t F_k(t_0- \eta,z) <0$ and $\partial_t F_k(t_0+ \eta,z) >0$ for any $z\in B(0,\eps)$. Thus $F_k$ also has a saddle-type geometry for $k$ large enough. Lemma A.1 of Thizy--V\'etois \cite{ThizyVetois} then applies and shows that, for $k$ large enough, $F_k$ admits a critical point $(t_k,z_k)$ in $(t_0-\eta,t_0+\eta) \times B(0,\eps)$. By Proposition \ref{LS}, $u_{k,t_k,z_k}$ is then a solution of \eqref{ynodal}. Standard arguments show that $u_k$ blows-up as $k \to  \infty$. Finally, by Propositions  \ref{DLenergie2} and \ref{DLenergie} and since $u_{k,t_k,z_k}$ solves \eqref{ynodal}, we have
\[ \bal \int_M |u_{k,t_k,z_k}|^{2^*} dv_{\tilde g}  = n I\(u_{k,t_k,z_k}\) & = K_n^{-n} + \int_M \check{u}_0^{2^*} dv_{\check g} + n \mu_k^{\pui} \( F(t_k,z_k) + \smallo\(1\) \) \\
 & < K_n^{-n} + \int_M \check{u}_0^{2^*} dv_{\check g}  
\eal \]
since $F(t_k,z_k) <0$. Since $\int_M \check{u}_0^{2^*} dv_{\check{g}} = Y(M, [g])^{\frac{n}{2}}$ by the $Y$-non-degeneracy assumption, this concludes the proof of Theorem \ref{Th1}. 
\end{proof}

\begin{remark}
If $\tilde g$ is defined as in \eqref{deftg}, it follows from an expansion at first-order in $x-y_k$ that $\tilde g$ satisfies  
\[ \Weyl_{\tilde g}(y_k) =- \eps_k W \]
for any $k\in\N$. In view of \eqref{DLWeyl}, Proposition \ref{DLenergie2} then shows that 
\begin{multline*}
I\(W_{k,t,0} \) = \frac{1}{n} K_n^{-n} + I(\check{u}_0) +  \Lambda(n) \check{u}_0(x_0) \mu_k(t)^{\pui} - C(n) |\Weyl_{\tilde g}(y_k)|_{\tilde g}^2 \mu_k(t)^4\\
 + \smallo\(\mu_k^{\pui}\),
\end{multline*}
an expansion that is reminiscent of those in Esposito--Pistoia--V\'etois \cite{EspositoPistoiaVetois} and Premoselli--Wei \cite{PremoselliWei}. 
\end{remark}

\section{Asymptotic analysis at the lowest sign-changing energy level and proof of Theorem \ref{Th2}} \label{Sec:linearpert}

We begin this section with a simple result showing that $Y\(\mathbb{S}^n,[g_0]\)^{\frac{n}{2}}+Y\(M,[g]\)^{\frac{n}{2}}$ is the minimal energy level for the blow-up of sign-changing solutions of \eqref{IntroEq1}:

\begin{proposition}\label{Pr1}
Let $\(M,g\)$ be a smooth, closed Riemannian manifold of dimension $n\ge3$ and positive Yamabe type. Let $\(u_k\)_{k\in\N}$ be a blowing-up sequence of solutions to \eqref{IntroEq1}. If $u_k$ changes sign for large $k$ and blows-up as $k\to\infty$, then 
\begin{equation}\label{Pr1Eq1}
E\(u_k\)\ge Y\(M,[g]\)^{\frac{n}{2}}+Y\(\mathbb{S}^n,[g_{std}]\)^{\frac{n}{2}} + \smallo\(1\)
\end{equation}
as $k\to\infty$. Moreover, if equality holds true in \eqref{Pr1Eq1} and $\(M,g\)$ is not conformally diffeomorphic to $\(\S^n,g_{std}\)$, then up to a subsequence and replacing $u_k$ by $-u_k$ if necessary, $u_k$ is of the form
\begin{equation}\label{IntroEq3}
u_k=u_0-\(\frac{\mu_k}{\mu_k^2+\frac{d_g\(\cdot,\xi_k\)^2}{n(n-2)}}\)^{\frac{n-2}{2}}+\smallo\(1\)\text{ in }H^1\(M\)\text{ as }k\to\infty,
\end{equation}
where $u_0$ is a positive solution to the Yamabe equation \eqref{IntroEq4} that attains $Y\(M,[g]\)$, $d_g$ is the geodesic distance with respect to the metric $g$ and $\(\mu_k\)_k$ and $\(\xi_k\)_k$ are two families of positive numbers and points in $M$, respectively, such that $\mu_k\to0$ and $\xi_k\to\xi_0$ as $k\to\infty$ for some point $\xi_0\in M$.
\end{proposition}

\proof[Proof of Proposition \ref{Pr1}]
We let
$$E:=\liminf_{k\to\infty}E\(u_k\)\quad\text{and}\quad E_0:=Y\(M,[g]\)^{\frac{n}{2}}+Y\(\mathbb{S}^n,[g_{std}]\)^{\frac{n}{2}}.$$ 
On the one hand, $Y\(M,[g]\)<Y\(\mathbb{S}^n,[g_{std}]\)$ when $\(M,g\)$ is not conformally diffeomorphic to $\(\S^n,g_{std}\)$ by the results of Trudinger \cite{Trudinger}, Aubin \cite{Aubin2} and Schoen \cite{SchoenYamabe}. On the other hand, it is easy to see that the energy of any sign-changing solutions to the Yamabe equation on $\(\mathbb{S}^n,g_{std}\)$ is greater than $2Y\(\mathbb{S}^n,[g_{std}]\)^{\frac{n}{2}}$. Struwe's decomposition \cite{Struwe} then shows that up to a subsequence and replacing $u_k$ by $-u_k$ if necessary,
\begin{equation*}
\left\{\begin{aligned}&\text{\eqref{IntroEq3} holds true with $u_0=0$ in $M$ if }E<E_0,\\&\text{\eqref{IntroEq3} holds true for some energy-minimizing solution $u_0$ of \eqref{IntroEq4} if }E=E_0\text{ and}\\&\text{$\(M,g\)$ is not conformally diffeomorphic to $\(\mathbb{S}^n,g_0\)$.}
\end{aligned}\right.
\end{equation*}
We are left to show that \eqref{IntroEq3} cannot hold true with $u_0=0$ in $M$. Assume by contradiction that this is the case. For large $k$, by testing \eqref{IntroEq1} with $u_k^+:=\max\(u_k,0\)$, we obtain 
\begin{equation}\label{Pr1Eq2}
\int_M\big(\left|\nabla u_k^+\right|^2+c_n\Scal_g\(u_k^+\)^2\big)dv_g=\int_M\(u_k^+\)^{2^*}dv_g.
\end{equation}
Since $u_k$ changes sign the right-hand side in \eqref{Pr1Eq2} is nonzero. An easy density argument then gives 
\begin{align}
\frac{\int_M\big(\left|\nabla u_k^+\right|^2+c_n\Scal_g\(u_k^+\)^2\big)dv_g}{\(\int_M\(u_k^+\)^{2^*}dv_g\)^{2/2^*}}&\ge\inf_{u\in C^\infty\(M\),\,u>0}\frac{\int_M\big(\left|\nabla u\right|^2+c_n\Scal_gu^2\big)dv_g}{\(\int_Mu^{2^*}dv_g\)^{2/2^*}}\nonumber\\
&=c_nY\(M,[g]\),\label{Pr1Eq3}
\end{align}
where $c_n := \frac{n-2}{4(n-1)}$. By using \eqref{Pr1Eq2} and \eqref{Pr1Eq3}, we then obtain
\begin{align}\label{Pr1Eq4}
c_nY\(M,[g]\)&\le\(\int_M\(u_k^+\)^{2^*}dv_g\)^{2/n}.
\end{align}
Since $Y\(M,[g]\)>0$, it then follows from \eqref{Pr1Eq4} that $u_k^+\not\to0$ in $L^{2^*}\(M\)$ as $k\to\infty$. This is in contradiction with \eqref{IntroEq3} when $u_0=0$ in $M$.
\endproof

The rest of this subsection is devoted to the proof of Theorem \ref{Th2}. By Lee and Parker's construction of conformal normal coordinates \cite{LeeParker} there exists a smooth family of metrics $\(g_\xi\)_{\xi\in M}$, $g_\xi=\Lambda_\xi^{2^*-2}g$, such that for every point $\xi\in M$, the function $\Lambda_\xi$ is smooth, positive and satisfies
\begin{equation}\label{Sec3Eq2}
\Lambda_\xi\(\xi\)=1,\quad\nabla\Lambda_\xi\(\xi\)=0\quad\text{and}\quad dv_{g_\xi}\(x\)=\big(1+\bigO\big(\left|x\right|^N\big)\big)dv_{\delta_0}\(x\)
\end{equation}
in geodesic normal coordinates, where $dv_{g_\xi}$ and $dv_{\delta_0}$ are the volume elements of the metric $g_\xi$ and the Euclidean metric $\delta_0$, respectively, and $N\in\N$ can be chosen arbitrarily large. In particular (see \cite{LeeParker}), it follows from \eqref{Sec3Eq2} that
\begin{equation}\label{Sec3Eq3}
\Scal_{g_\xi}(\xi)=0,\quad\nabla \Scal_{g_\xi}(\xi)=0\quad\text{and}\quad\Delta_g \Scal_{g_\xi}(\xi)=\frac{1}{6}|\Weyl_g(\xi)|_g^2.
\end{equation}
Also, if $g$ is locally conformally flat in a neighbourhood of $\xi \in M$, then $g_{\xi}$ can be chosen flat in a neighbourhood of $\xi$. We first state some preliminary lemmas that apply to the more general case where the weak limit $u_0$ is a possibly sign-changing solution of
\begin{equation}\label{Sec4Eq}
\Delta_gu_0+c_n\Scal_gu_0=\left|u_0\right|^{2^*-2}u_0\quad\text{in }M.
\end{equation}

The following lemma is essentially contained in Premoselli \cite{Premoselli13} (see also Druet--Hebey--Robert \cite{DruetHebeyRobert} in the case of positive solutions):

\begin{lemma}\label{Lem1}
Let $\(M,g\)$ be a smooth, closed Riemannian manifold of dimension $n\ge3$ and $u_0$ be a solution to \eqref{Sec4Eq}. Assume that there exists a sequence of solutions $\(u_k\)_{k\in\N}$ of type \eqref{IntroEq3} to \eqref{IntroEq1}. Let $\(g_\xi\)_{\xi\in M}$, $g_\xi=\Lambda_\xi^{2^*-2}g$, be the Lee--Parker family of conformal metrics to $g$. Let $\(\overline\mu_k\)_{k\in\N}$ and $\(\overline\xi_k\)_{k\in\N}$ be families of positive numbers and points in $M$, respectively, such that 
\begin{equation}\label{Lem1Eq1}
u_k\(\overline\xi_k\)=\min_M u_k=- \overline\mu_k^{1 - \frac{n}{2}}.
\end{equation}
Then 
\begin{equation}\label{Lem1Eq2}
u_k-u_0+B_k=\smallo\big(B_k+1\big)
\end{equation}
as $k\to\infty$, uniformly in $M$, where $B_k$ is defined as
$$B_k\(x\):=\Lambda_{\overline\xi_k}\(x\)\frac{\overline\mu_k^{\frac{n-2}{2}}}{\Big( \overline\mu_k^2+\frac{d_{g_{\overline\xi_k}}\(x,\overline\xi_k\)^2}{n(n-2)}\Big)^{\frac{n-2}{2}}}\quad\forall x\in M.$$ 
\end{lemma}

\proof[Proof of Lemma \ref{Lem1}]
Since $\(u_k\)_{k}$ is of the form \eqref{IntroEq3}, the main result in Premoselli \cite{Premoselli13} shows that 
\begin{equation}\label{Lem1Eq3}
u_k=u_0-\frac{\mu_k^{\frac{n-2}{2}}}{\( \mu_k^2+\frac{d_g\(x,\xi_k\)^2}{n(n-2)}\)^{\frac{n-2}{2}}}+\smallo\(1+\(\frac{\mu_k}{\mu_k^2+d_g\(\cdot,\xi_k\)^2}\)^{\frac{n-2}{2}}\)
\end{equation}
as $k\to\infty$, uniformly in $M$, where $\(\mu_k\)_k$ and $\(\xi_k\)_k$ are as in \eqref{IntroEq3}. Let $\(\overline\mu_k\)_{k}$ and $\(\overline\xi_k\)_{k}$ be such that \eqref{Lem1Eq1} holds true. It then follows from \eqref{Lem1Eq3} that 
\begin{equation}\label{Lem1Eq4}
\overline\mu_k\sim\mu_k\quad\text{ and }\quad d_g\(\overline\xi_k,\xi_k\)=\smallo\(\mu_k\)
\end{equation}
as $k\to\infty$. Moreover, since $\Lambda_{\overline\xi_k}\(\overline\xi_k\)=1$, we obtain 
\begin{equation}\label{Lem1Eq5}
d_{g_{\overline\xi_k}}\(x,\overline\xi_k\)^2=d_g\(x,\overline\xi_k\)^2+\bigO\(d_g\(x,\overline\xi_k\)^3\)
\end{equation}
uniformly with respect to $k\in\N$ and $x\in M$. By using again that $\Lambda_{\overline\xi_k}\(\overline\xi_k\)=1$ together with \eqref{Lem1Eq4} and \eqref{Lem1Eq5}, we obtain
\begin{equation}\label{Lem1Eq6}
\frac{\mu_k^{\frac{n-2}{2}}}{\( \mu_k^2+\frac{d_{g}\(x,\xi_k\)^2}{n(n-2)}\)^{\frac{n-2}{2}}}=B_k+\smallo\big(B_k+1\big)
\end{equation}
as $k\to\infty$, uniformly in $M$, so that \eqref{Lem1Eq2} follows from \eqref{Lem1Eq3} and \eqref{Lem1Eq6}.
\endproof

We then state the following result from Premoselli--V\'etois \cite{PremoselliVetois2}, which essentially follows from a Pohozaev-type identity together with the estimate \eqref{Lem1Eq3}:

\begin{lemma}\label{Lem2}(Lemma 3.2 in \cite{PremoselliVetois2})
Let $\(M,g\)$ be a smooth, closed Riemannian manifold of dimension $n\ge3$ and $u_0$ be a solution to \eqref{Sec4Eq}. Assume that there exists a sequence of solutions $\(u_k\)_{k\in\N}$ to \eqref{IntroEq1} satisfying \eqref{IntroEq3}. Let $\(\overline\mu_k\)_k$, $\(\overline\xi_k\)_k$, $\(B_k\)_k$ and $\(g_\xi\)_\xi$ be as in Lemma \ref{Lem1}. Then
\begin{multline}\label{Lem2Eq1}
\int_{B\(0,1/\sqrt{\overline\mu_k}\)}\(\<\nabla\hat{u}_k,\cdot\>_{\delta_0}+\frac{n-2}{2}\hat{u}_k\)\(\(\Delta_{\hat{g}_k}-\Delta_{\delta_0}\)\hat{u}_k+\overline\mu_k^2\hat{h}_k\hat{u}_k\)dv_{\hat{g}_k}\\
=\frac{1}{2}n^{\pui} (n-2)^{\frac{n+2}{2}} \omega_{n-1}u_0\(\xi_0\)\overline\mu_k^{\frac{n-2}{2}}+\smallo\(\overline\mu_k^{\frac{n-2}{2}}\)
\end{multline}
as $k\to\infty$, where 
\begin{equation}\label{Lem2Eq2}
\hat{u}_k\(y\):=\overline\mu_k^{\frac{n-2}{2}}\big(\Lambda_{\overline\xi_k}^{-1}u_k\big)\big(\exp_{\overline\xi_k}\(\overline\mu_k y\)\big),\quad\hat{g}_k\(y\):=\exp_{\overline\xi_{k}}^*g_{\overline\xi_{k}}\(\overline \mu_k y\)
\end{equation}
and
\begin{equation}\label{Lem2Eq3}
\hat{h}_k\(y\):=c_n\Scal_{\hat{g}_k}\big(\exp_{\overline\xi_k}\(\overline\mu_k y\)\big)
\end{equation}
for all points $y\in B\(0,1/\sqrt{\overline\mu_k}\)$, $\omega_{n-1}$ is the volume of the standard unit sphere of dimension $n-1$, $\delta_0$ is the Euclidean metric in $\R^n$, $\exp_{\overline\xi_k}$ is the exponential map at the point $\overline\xi_k$ with respect to the metric $g_{\overline\xi_k}$ and we identify $T_{\overline\xi_k}M$ with $\R^n$.
\end{lemma}

Finally, to prove Theorem \ref{Th2} in dimensions $n\ge7$, we need the following refinement of Lemma \ref{Lem1}:

\begin{lemma}\label{Lem3}
Let $\(M,g\)$ be a smooth, closed Riemannian manifold of dimension $n\ge7$ and $u_0$ be a solution to \eqref{Sec4Eq}. Let $\(u_k\)_k$, $\(\overline\mu_k\)_k$, $\(\overline\xi_k\)_k$, $\(B_k\)_k$ and $\(g_\xi\)_\xi$ be as in Lemma \ref{Lem1}. Then, for $i\in\left\{0,1,2\right\}$,
\begin{align}\label{Lem3Eq1}
&\left|\nabla^i\(u_k-u_0+B_k\)\(x\)\right|\nonumber\\
&=\bigO\(\left\{\begin{aligned}&\frac{1}{\(\overline\mu_k+d_{g_{\overline\xi_k}}\(x,\overline\xi_k\)\)^{i}}&&\text{if }\(M,g\)\text{ is l.c.f.}\\&\frac{1}{\(\overline\mu_k+d_{g_{\overline\xi_k}}\(x,\overline\xi_k\)\)^{i}}+\frac{\overline\mu_k^{\frac{n-2}{2}}}{\(\overline\mu_k+d_{g_{\overline\xi_k}}\(x,\overline\xi_k\)\)^{n-6+i}}&&\text{otherwise}\end{aligned}\right\}\)
\end{align}
uniformly with respect to $k\in\N$ and $x\in M$, where l.c.f. stands for locally conformally flat.
\end{lemma}

\proof[Proof of Lemma \ref{Lem3}]
We only need to prove \eqref{Lem3Eq1} when $i=0$, since the cases $i=1$ and $i=2$ follow by standard interior estimates. Remark that when $x$ satisfies $d_{g_{\overline\xi_k}}\(x,\overline\xi_k\)\ge\sqrt{\overline\mu_k}$, the estimate \eqref{Lem3Eq1} when $i=0$ follows from \eqref{Lem1Eq2}. Therefore, it suffices to consider the case where $d_{g_{\overline\xi_k}}\(x,\overline\xi_k\)<\sqrt{\overline\mu_k}$ as $k\to\infty$. In this case, we use an approach inspired from the work of Chen--Lin \cite{ChenLin} in the case of positive solutions that was applied to the Yamabe equation by Marques \cite{Marques}. By letting  $x=\exp_{\overline\xi_k}\(\overline\mu_k y\)$ for $y\in B\(0,1/\sqrt{\overline\mu_k}\)$, it is easily seen that \eqref{Lem3Eq1} is implied by the following estimate:
\begin{multline}\label{Lem3Eq2}
\left|\nabla^i\(\hat{u}_k+B_0\)\(y\)\right|=\bigO\(\left\{\begin{aligned}&\frac{\overline\mu_k^{\frac{n-2}{2}}}{\(1+\left|y\right|\)^{i}}&&\text{if }\(M,g\)\text{ is l.c.f.}\\&\frac{\overline\mu_k^{\frac{n-2}{2}}}{\(1+\left|y\right|\)^{i}}+\frac{\overline\mu_k^4}{\(1+\left|y\right|\)^{n-6+i}}&&\text{otherwise}\end{aligned}\right\}\),
\end{multline}
where $\hat{u}_k$ is as in \eqref{Lem2Eq2} and we have let 
\begin{equation}\label{Sec3Eq6}
B_0\(y\):=\(1+\frac{\left|y\right|^2}{n(n-2)}\)^{-\frac{n-2}{2}}\quad\forall y\in\R^n.
\end{equation}
As mentioned before, we only have to prove \eqref{Lem3Eq2} for $i=0$. By conformal invariance of the conformal Laplacian, the equation \eqref{IntroEq1} can be rewritten as 
\begin{equation}\label{Lem3Eq3}
\Delta_{\hat{g}_k}\hat{u}_k+\overline\mu_k^2\hat{h}_k\hat{u}_k=\left|\hat{u}_k\right|^{2^*-2}\hat{u}_k\quad\text{in }B\big(0,1/\sqrt{\overline\mu_k}\big),
\end{equation}
where $\hat{h}_k$ and $\hat{g}_k$ are as in \eqref{Lem2Eq2} and \eqref{Lem2Eq3}. We let $y_k\in\overline{B\big(0,1/\sqrt{\overline\mu_k}\big)}$ be such that 
\begin{equation}\label{Lem3Eq4}
\left|\(\hat{u}_k+B_0\)\(y_k\)\right|=\max_{\overline{B(0,1/\sqrt{\overline\mu_k})}}\left|\(\hat{u}_k+B_0\)\right|=:\lambda_k,
\end{equation}
and we define
$$\psi_k:=\lambda_k^{-1}\(\hat{u}_k+B_0\).$$
By \eqref{Lem1Eq2}, $\lambda_k = \smallo(1)$ as $k\to\infty$.  By using \eqref{Lem3Eq3} together with the equation $\Delta_{\delta_0}B_0=B_0^{2^*-1}$, we obtain
\begin{equation}\label{Lem3Eq5}
\Delta_{\hat{g}_k}\psi_k=\(2^*-1\)B_0^{2^*-2}\psi_k+\lambda_k^{-1}f_k\quad\text{in }B\big(0,1/\sqrt{\overline\mu_k}\big),
\end{equation}
where
\begin{equation}\label{Lem3Eq6}
f_k:=\(\Delta_{\hat{g}_k}-\Delta_{\delta_0}\)B_0-\overline\mu_k^2\hat{h}_k\hat{u}_k+\left|\hat{u}_k\right|^{2^*-2}\hat{u}_k+B_0^{2^*-1}-\(2^*-1\)\lambda_k B_0^{2^*-2}\psi_k.
\end{equation}
We now estimate the terms in the right-hand side of \eqref{Lem3Eq6}. Since $B_0$ is radially symmetric, it follows from \eqref{Sec3Eq2} that 
\begin{equation}\label{Lem3Eq7}
\(\Delta_{\hat{g}_k}-\Delta_{\delta_0}\)B_0\(y\)=\bigO\big(\overline\mu_k^{N}\left|y\right|^{N-1}\left|\nabla B_0\(y\)\right|\big)=\bigO\(\frac{\overline\mu_k^{N}\left|y\right|^{N}}{\(1+\left|y\right|\)^n}\)
\end{equation}
uniformly with respect to $k\in\N$ and $y\in B\(0,1/\sqrt{\overline\mu_k}\)$. By using \eqref{Sec3Eq3} and \eqref{Lem1Eq2}, we obtain 
\begin{equation}\label{Lem3Eq8}
\hat{h}_k\(y\)\hat{u}_k\(y\)=\left\{\begin{aligned}&0&&\text{if }\(M,g\)\text{ is l.c.f.}\\&\bigO\(\frac{\overline\mu_k^2\left|y\right|^2}{\(1+\left|y\right|\)^{n-2}}\)&&\text{otherwise}\end{aligned}\right.
\end{equation}
uniformly with respect to $k\in\N$ and $y\in B\(0,1/\sqrt{\overline\mu_k}\)$. Since $2^*-2<2$ when $n\ge7$, by using \eqref{Lem1Eq2} together with straightforward estimates, we obtain
\begin{equation}\label{Lem3Eq9}
\left|\hat{u}_k\right|^{2^*-2}\hat{u}_k+B_0^{2^*-1}-\(2^*-1\)\lambda_k B_0^{2^*-2}\psi_k= \smallo\(\lambda_k B_0^{2^*-2}\left|\psi_k\right|\)
\end{equation}
as $k\to\infty$, uniformly with respect to $y\in B\(0,1/\sqrt{\overline\mu_k}\)$. Let $G_k$ be the Green's function of the operator $\Delta_{\hat{g}_k}$ in $B\(0,1/\sqrt{\overline\mu_k}\)$ with zero Dirichlet boundary condition. Then
\begin{multline}\label{Lem3Eq11}
\psi_k\(y\)=\int_{B\(0,1/\sqrt{\overline\mu_k}\)}G_k\(y,\cdot\)\(\(2^*-1\)B_0^{2^*-2}\psi_k+\lambda_k^{-1}f_k\)dv_{\hat{g}_k}\\
-\int_{\partial B\(0,1/\sqrt{\overline\mu_k}\)}\partial_\nu G_k\(y,\cdot\)\psi_k\,d\sigma_{\hat{g}_k}
\end{multline}
for all points $x\in B\(0,1/\sqrt{\overline\mu_k}\)$.
Standard estimates on Green's functions (see e.g. Robert \cite{RobDirichlet}) give that for large $k$, 
\begin{equation}\label{Lem3Eq12}
G_k\(y,z\)=\bigO\big(\left|y-z\right|^{2-n}\big)\quad\text{for }z\in B\big(0,1/\sqrt{\overline\mu_k}\big)
\end{equation}
and
\begin{equation}\label{Lem3Eq13}
\partial_\nu G_k\(y,z\)=\bigO\big(\left|y-z\right|^{1-n}\big)\quad\text{for }z\in\partial B\big(0,1/\sqrt{\overline\mu_k}\big)
\end{equation}
uniformly with respect to $k\in\N$ and $y\in B\(0,1/\sqrt{\overline\mu_k}\)$. Remark also that \eqref{Lem1Eq2} gives
\begin{equation}\label{Lem3Eq14}
\psi_k=\bigO\(\lambda_k^{-1}\overline\mu_k^{\frac{n-2}{2}}\)\,\text{as }k\to\infty\text{, uniformly in } B\big(0,1/\sqrt{\overline\mu_k}\big) \backslash B\big(0,1/2\sqrt{\overline\mu_k}\big) \ .
\end{equation}
By using \eqref{Lem3Eq7}, \eqref{Lem3Eq8}, \eqref{Lem3Eq9} and \eqref{Lem3Eq11}--\eqref{Lem3Eq14} together with straightforward integral estimates and the fact that $n\ge7$, we obtain
\begin{multline}\label{Lem3Eq15}
\psi_k\(y\)=\bigO\Bigg(\int_{B\(0,1/\sqrt{\overline\mu_k}\)}\frac{\left|\psi_k\(z\)\right|dz}{\left|y-z\right|^{n-2}\(1+\left|z\right|\)^4}+\lambda_k^{-1}\overline\mu_k^{\frac{n-2}{2}}\\
\left\{+\frac{\lambda_k^{-1}\overline\mu_k^4}{\(1+\left|y\right|\)^{n-6}}\quad\text{if }\(M,g\)\text{ is not l.c.f.}\right\}\Bigg)
\end{multline}
uniformly with respect to $k\in\N$ and $y\in B\(0,1/\sqrt{\overline\mu_k}\)$. We now claim that 
\begin{equation}\label{Lem3Eq16}
\lambda_k=\bigO\(\overline\mu_{k}^{\frac{n-2}{2}}\left\{+\overline\mu_{k}^4\quad\text{if }\(M,g\)\text{ is not l.c.f.}\right\}\)
\end{equation}
uniformly with respect to $k\in\N$. Assume by contradiction that \eqref{Lem3Eq16} does not hold true, i.e. there exists a subsequence $\(k_j\)_{j\in\N}$ of positive numbers such that
\begin{equation}\label{Lem3Eq17}
k_j\to0\quad\text{and}\quad\overline\mu_{k_j}^{\frac{n-2}{2}}\big\{+\overline\mu_{k_j}^4\quad\text{if }\(M,g\)\text{ is not l.c.f.}\big\}=\smallo\(\lambda_{k_j}\)
\end{equation}
as $j\to\infty$. Since $\left|\psi_{k_j}\right|\le1$ in $B\big(0,1/\sqrt{\overline\mu_{k_j}}\big)$ and $\hat{g}_{k_j}\to\delta_0$ as $j \to  \infty$, uniformly in compact subsets of $\R^n$, it follows from \eqref{Lem3Eq5}--\eqref{Lem3Eq9} together with standard elliptic estimates that up to a subsequence, $\(\psi_{k_j}\)_j$ converges in $C^1_{\loc}\(\R^n\)$ as $j\to\infty$ to a solution $\psi_0\in C^2\(\R^n\)$ of the equation
\begin{equation}\label{Lem3Eq10}
\Delta_{\delta_0}\psi_0=\(2^*-1\)B_0^{2^*-2}\psi_0\quad\text{in }\R^n.
\end{equation}
Independently, it follows from \eqref{Lem3Eq15} and \eqref{Lem3Eq17} that 
\begin{equation}\label{Lem3Eq18}
\psi_{k_j}\(y\)=\bigO\big(\(1+\left|y\right|\)^{-2}\big)+\smallo\(1\)
\end{equation}
as $j\to\infty$, uniformly with respect to $y\in B\big(0,1/\sqrt{\overline\mu_{k_j}}\big)$. Passing to the limit into \eqref{Lem3Eq18} as $j\to\infty$, we then obtain
\begin{equation}\label{Lem3Eq19}
\psi_0\(y\)=\bigO\big(\(1+\left|y\right|\)^{-2}\big)
\end{equation}
uniformly with respect to $y\in \R^n$. By applying Lemma 2.4 of Chen--Lin \cite{ChenLinMP}, it follows from \eqref{Lem3Eq10} and \eqref{Lem3Eq19} that
$$\psi_0\(y\)=c_0\frac{1-\frac{\left|y\right|^2}{n(n-2)}}{\(1+\frac{\left|y\right|^2}{n(n-2)}\)^{\frac{n}{2}}}+\sum_{i=1}^nc_i\frac{y_i}{\(1+\frac{\left|y\right|^2}{n(n-2)}\)^{\frac{n}{2}}}\quad\forall y\in\R^n$$
for some numbers $c_0,\dotsc,c_n\in\R$. On the other hand, it follows from \eqref{Lem1Eq1} that $\psi_k\(0\)=\left|\nabla\psi_k\(0\)\right|=0$, which gives $\psi_0\(0\)=\left|\nabla\psi_0\(0\)\right|=0$. Therefore, we obtain $c_0=\dotsb=c_n=0$, and so $\psi_0=0$ in $\R^n$. Since $\psi_{k_j}\(y_{k_j}\)=1$ for all $j\in\N$, where $y_{k_j}$ is as in \eqref{Lem3Eq4}, we then obtain that $\left|y_{k_j}\right|\to\infty$ as $j\to\infty$, which is in contradiction with \eqref{Lem3Eq18}. This proves that \eqref{Lem3Eq16} holds true. Finally, by using \eqref{Lem3Eq16} together with successive applications of \eqref{Lem3Eq15}, we obtain that 
\begin{equation}\label{Lem3Eq20}
\psi_k\(y\)=\bigO\Bigg(\frac{\ln\(2+\left|y\right|\)}{\(1+\left|y\right|\)^{n-2}}+\lambda_k^{-1}\overline\mu_k^{\frac{n-2}{2}}\left\{+\frac{\lambda_k^{-1}\overline\mu_k^4}{\(1+\left|y\right|\)^{n-6}}\quad\text{if }\(M,g\)\text{ is not l.c.f.}\right\}\Bigg)
\end{equation}
uniformly with respect to $k\in\N$ and $y\in B\big(0,1/\sqrt{\overline\mu_k}\big)$. The estimate \eqref{Lem3Eq2} with $i=0$ then follows from \eqref{Lem3Eq16} and \eqref{Lem3Eq20}.
\endproof

We can now use Lemmas \ref{Lem2} and \ref{Lem3} to prove Theorem \ref{Th2}.

\proof[Proof of Theorem \ref{Th2}]
Assume by contradiction that 
$$E(M, [g])\le Y(\mathbb{S}^n,[g_{std}])^{\frac{n}{2}}+Y(M,[g])^{\frac{n}{2}}.$$ 
By using the definition of $E(M, [g])$ together with a diagonal argument, it follows that there exists a sequence of blowing-up solutions $\(u_k\)_{k\in\N}$ to \eqref{IntroEq1} such that
$$E\(u_k\)\le Y\(M,[g]\)^{\frac{n}{2}}+Y\(\mathbb{S}^n,[g_{std}]\)^{\frac{n}{2}}+\smallo\(1\)$$
as $k\to\infty$. In the case where the functions $u_k$ do not change sign, we obtain a contradiction with the assumptions of Theorem \ref{Th2} by applying the compactness results for positive solutions of the Yamabe equation (see Schoen \cite{SchoenPreprint,Schoenlcf}, Li--Zhu \cite{LiZhu}, Druet \cite{DruetYlowdim}, Marques \cite{Marques}, Li--Zhang \cite{LiZhang1,LiZhang2} and Khuri--Marques--Schoen \cite{KhuMaSc}). Therefore, in what follows, we assume that the functions $u_k$ change sign. It then follows from Proposition~\ref{Pr1} that up to a subsequence and replacing $u_k$ by $-u_k$ if necessary, the functions $u_k$ are of the form \eqref{IntroEq3} for some energy-minimizing, positive solution $u_0$ to \eqref{IntroEq4}. In the case where $n\le6$, we can apply a more general compactness result that we obtained in Premoselli--V\'etois \cite{PremoselliVetois2} (Theorem 1.2 in \cite{PremoselliVetois2}). Therefore, in what follows, we assume that $n\ge7$. By using \eqref{Lem3Eq1} (see also \eqref{Lem3Eq2}) and \eqref{Lem3Eq7}, we then obtain
\begin{align}\label{Th3Eq1}
\(\Delta_{\hat{g}_k}-\Delta_{\delta_0}\)\hat{u}_k&=\bigO\Bigg(\frac{\overline\mu_k^{N}\left|y\right|^{N}}{\(1+\left|y\right|\)^n}+\left|\(\hat{g}_k-\delta_0\)\(y\)\right||\nabla^2\(\hat{u}_k+B_0\)\(y\)|\nonumber\\
&\qquad+\left|\nabla\(\hat{g}_k-\delta_0\)\(y\)\right|\left|\nabla\(\hat{u}_k+B_0\)\(y\)\right|\Bigg)\nonumber\allowdisplaybreaks\\
&=\bigO\(\overline\mu_k^{\frac{n+2}{2}}\left\{+\frac{\overline\mu_k^6\left|y\right|}{\(1+\left|y\right|\)^{n-4}}\quad\text{ if }\(M,g\)\text{ is not l.c.f.}\right\}\)
\end{align}
uniformly with respect to $k\in\N$ and $y\in B\(0,1/\sqrt{\overline\mu_k}\)$. It then follows from \eqref{Lem3Eq1}, \eqref{Lem3Eq8} and \eqref{Th3Eq1} that
\begin{multline}\label{Th3Eq2}
\int_{B\(0,1/\sqrt{\overline\mu_k}\)}\(\<\nabla\hat{u}_k,\cdot\>_{\delta_0}+\frac{n-2}{2}\hat{u}_k\)\(\(\Delta_{\hat{g}_k}-\Delta_{\delta_0}\)\hat{u}_k+\overline\mu_k^2\hat{h}_k\hat{u}_k\)dv_{\delta_0}\\
-\overline\mu_k^2\int_{B\(0,1/\sqrt{\overline\mu_k}\)}\(\<\nabla B_0,\cdot\>_{\delta_0}+\frac{n-2}{2}B_0\)\hat{h}_k B_0\,dv_{\delta_0}\\
=\bigO\(\left\{\begin{aligned}&\overline\mu_k^{\frac{n}{2}}&&\text{if }\(M,g\)\text{ is l.c.f.}\\&\overline\mu_k^{\frac{n}{2}}+\overline\mu_k^6&&\text{otherwise}\end{aligned}\right\}\)
\end{multline}
for large $k$. On the other hand, by using \eqref{Sec3Eq2} and \eqref{Sec3Eq3} together with straightforward computations and symmetry arguments, we obtain
\begin{multline}\label{Th3Eq3}
\int_{B\(0,1/\sqrt{\overline\mu_k}\)}\(\<\nabla B_0,\cdot\>_{\delta_0}+\frac{n-2}{2}B_0\)\hat{h}_k B_0\,dv_{\delta_0}\\
=\left\{\begin{aligned}&0&&\text{if }\(M,g\)\text{ is l.c.f.}\\&a_n\left|\Weyl_g\(\xi_0\)\right|_g^2\overline\mu_k^2+\smallo\(\overline\mu_k^2\)&&\text{otherwise}\end{aligned}\right.
\end{multline}
as $k\to\infty$, where
\begin{equation}\label{an}
a_n:=\frac{c_n}{24}(n-2)^2\int_{\R^n}\bigg(1+\frac{\left|y\right|^2}{n(n-2)}\bigg)^{1-n}\bigg(\frac{\left|y\right|^2}{n(n-2)}-1\bigg)\frac{\left|y\right|^2dy}{n(n-2)}.
\end{equation}
The constant $a_n$ is computed by using \eqref{indu} and \eqref{In}, and we obtain 
\begin{align}\label{Th3Eq4}
a_n=\frac{c_n}{6} \frac{n^{\frac{n}{2}}(n-2)^{\frac{n+4}{2}} }{n-6} \omega_{n-1} I_{n-1}^{\frac{n}{2}}&=\frac{n^{\frac{n+2}{2}}(n-2)^{\frac{n+4}{2}}}{12(n-6)(n-4)}\omega_{n-1}I_{n}^{\frac{n-2}{2}}\nonumber\\
&=\frac{n(n-2)^2}{6(n-6)(n-4)}K_n^{-n}.
\end{align}
By putting together \eqref{Lem2Eq1}, \eqref{Th3Eq2} and \eqref{Th3Eq3}, we then obtain
\begin{multline}\label{Th3Eq5}
\frac12 n^{\pui} (n-2)^{\frac{n+2}{2}}\omega_{n-1}u_0\(\xi_0\)\overline\mu_k^{\frac{n-6}{2}}-a_n\left|\Weyl_g\(\xi_0\)\right|_g^2\overline\mu_k^2\\
=\smallo\(\left\{\begin{aligned}&\overline\mu_k^{\frac{n-6}{2}}&&\text{if }\(M,g\)\text{ is l.c.f.}\\&\overline\mu_k^2+\overline\mu_k^{\frac{n-6}{2}}&&\text{otherwise}\end{aligned}\right\}\)
\end{multline}
as $k\to\infty$. In the case where $n=10$, we obtain
\begin{equation}\label{Th3Eq6}
2\cdot10^{-4} 8^{-6} a_{10}=\frac{5}{567}\,\omega_9.
\end{equation}
Finally, by using \eqref{Th3Eq5} and \eqref{Th3Eq6}, we obtain a contradiction with the assumptions of Theorem \ref{Th2}.
\endproof

\begin{remark}
The approach used in the proof of Theorem \ref{Th2} can be extended to the case of sequences of solutions $\(u_k\)_{k\in\N}$ of type \eqref{IntroEq3} to linear perturbations of the Yamabe equation of the form 
\begin{equation}\label{IntroEq2}
\Delta_gu_k+\(c_n\Scal_g+\eps_k h\)u_k=\left|u_k\right|^{2^*-2}u_k\quad\text{in }M,
\end{equation}
where $h\in C^{0,\vartheta}\(M\)$, $\vartheta\in\(0,1\)$, and $\(\eps_k\)_k$ is a sequence of positive real numbers such that $\eps_k\to0$ as $k\to\infty$. In the case of positive solutions, perturbed equations of the form \eqref{IntroEq3} have been studied for example by Esposito--Pistoia--V\'etois \cite{EspositoPistoiaVetois}, Morabito--Pistoia--Vaira \cite{MorabitoPistoiaVaira}, Premoselli \cite{Premoselli12} and Robert--V\'etois \cite{RobertVetois5}. In this case, in place of \eqref{Lem3Eq1}, we obtain, for $i\in\left\{0,1,2\right\}$,
\begin{align*}
&\left|\nabla^i\(u_k-u_0+B_k\)\(x\)\right|\\
&=\bigO\(\left\{\begin{aligned}&\frac{1}{\(\overline\mu_k+d_{g_{\overline\xi_k}}\(x,\overline\xi_k\)\)^{i}}+\frac{\eps_k\overline\mu_k^{\frac{n-2}{2}}}{\(\overline\mu_k+d_{g_{\overline\xi_k}}\(x,\overline\xi_k\)\)^{n-4+i}}\quad\text{if }\(M,g\)\text{ is l.c.f.}\\&\frac{1}{\(\overline\mu_k+d_{g_{\overline\xi_k}}\(x,\overline\xi_k\)\)^{i}}+\frac{\overline\mu_k^{\frac{n-2}{2}}\(\eps_k+\overline\mu_k^2+d_{g_{\overline\xi_k}}\(x,\overline\xi_k\)^2\)}{\(\overline\mu_k+d_{g_{\overline\xi_k}}\(x,\overline\xi_k\)\)^{n-4+i}}\quad\text{otherwise}\end{aligned}\right\}\)
\end{align*}
uniformly with respect to $k\in\N$ and $x\in M$, and in place of \eqref{Th3Eq5}, we obtain
\begin{multline*}
\frac12 n^{\pui} (n-2)^{\frac{n+2}{2}}\omega_{n-1}u_0\(\xi_0\)\overline\mu_k^{\frac{n-6}{2}}-a_n\left|\Weyl_g\(\xi_0\)\right|_g^2\overline\mu_k^2+b_n\eps_k h\(\xi_0\)\\
=\smallo\(\left\{\begin{aligned}&\eps_k+\overline\mu_k^{\frac{n-6}{2}}&&\text{if }\(M,g\)\text{ is l.c.f.}\\&\eps_k+\overline\mu_k^2+\overline\mu_k^{\frac{n-6}{2}}&&\text{otherwise}\end{aligned}\right\}\)
\end{multline*}
as $k\to\infty$, where $a_n$ is as in \eqref{an} and $b_n$ is another positive constant depending only on $n$. This yields that there does not exist any sign-changing blowing-up sequences of solutions of type \eqref{IntroEq3} to \eqref{IntroEq2} in each of the following situations:
\begin{itemize}
\item $3\le n\le 6$,
\item $h\ge0$ and $7\le n\le 9$,
\item $h\ge0$, $n=10$ and $u_0>\frac{5}{567}|\Weyl_g|^2_g$ for all points in $M$,
\item $h\le0$, $n=10$ and $u_0<\frac{5}{567}|\Weyl_g|^2_g$ for all points in $M$,
\item $h\ge0$, $n\ge11$ and $\(M,g\)$ is locally conformally flat,
\item $h\le0$, $n\ge11$ and $\Weyl_g\ne0$ for all points in $M$. 
\end{itemize}
Conversely, it is not difficult to adapt the constructive proofs in Esposito--Pistoia--V\'etois \cite{EspositoPistoiaVetois} and Robert--V\'etois \cite{RobertVetois5} to prove that if $u_0$ is a nondegenerate positive solution to \eqref{IntroEq4}, then there exists a blowing-up sequence of solutions of type \eqref{IntroEq3} to \eqref{IntroEq2} in each of the following situations:
\begin{itemize}
\item $\ds\min_Mh<0$ and $7\le n\le 9$,
\item $\ds\min_Mh<0$, $n=10$ and $u_0>\frac{5}{567}|\Weyl_g|^2_g$  for all points in $M$,
\item $\ds\max_Mh>0$, $n=10$ and $u_0<\frac{5}{567}|\Weyl_g|^2_g$ for all points in $M$,
\item $\ds\min_Mh<0$, $n\ge11$ and $\(M,g\)$ is locally conformally flat,
\item $\ds\max_Mh>0$, $n\ge11$ and $\Weyl_g\ne0$ for all points in $M$. 
\end{itemize}
This is again in sharp contrast with the case of positive solutions (see Esposito--Pistoia--V\'etois \cite{EspositoPistoiaVetois}), where blowing-up sequences of solutions to \eqref{IntroEq2} exist when $\ds\max_Mh>0$ in all dimensions $n\ge4$.
\end{remark}

\bibliographystyle{amsplain}
\bibliography{biblio}

\end{document}